\documentclass[12pt]{article}
\usepackage{euscript,amsmath, amssymb, amsfonts}
\usepackage{color}

\newcommand{\ig}{\Big}
\newcommand{\igg}{\Bigg}

\newcommand{\sigmama}{ {{g}} }
\newcommand{\p}{\partial}
\newcommand{\R}{{\cal R}}
\newcommand{\HH}{ H_{1,\, \eta}(G) }
\newcommand{\G}{{\mathbb C}[G]}
\newcommand{\str}{sp }

\newcommand{\ce}{{\cal Z}^L }
\newcommand{\se}{{\cal Z}^S }
\newcommand{\LLL}{ L_{1,\, \eta}(G) }
\newcommand{\SSS}{ S_{1,\, \eta}(G) }

\newcommand{\be}{\begin{equation}}
\newcommand{\ee}{\end{equation}}
\newcommand{\bee}{\begin{eqnarray}}
\newcommand{\eee}{\end{eqnarray}}
\newcommand\nn{\nonumber \\}
\newcommand\n{ }
\newcommand\defeq{:=}

\newcounter{theorem}

\makeatletter \@addtoreset{theorem}{section}

\newcounter{lemma}

\makeatletter \@addtoreset{lemma}{section}

\newcounter{proposition}

\makeatletter \@addtoreset{proposition}{section}

\newcounter{conjecture}

\makeatletter \@addtoreset{conjecture}{section}

\newcounter{definition}

\makeatletter \@addtoreset{definition}{section}

\newenvironment{proof}[1][Proof]{\noindent\textsf{#1.\ }}
{\hfill {\small $\square$}}

\makeatletter \@addtoreset{equation}{section}

\begin{document}


\sloppy \title
 {
The number of independent Traces and Supertraces on
Symplectic Reflection Algebras
}
\author
 {
 S.E.Konstein\thanks{E-mail: konstein@lpi.ru}\ \ and
 I.V.Tyutin\thanks{E-mail: tyutin@lpi.ru}
 \thanks{
               This work was supported
               by the RFBR (grant No.~11-02-00685)
 } \\
               {\sf \small I.E.Tamm Department of
               Theoretical Physics,} \\ {\sf \small P. N. Lebedev Physical
               Institute,} \\ {\sf \small 119991, Leninsky Prospect 53,
               Moscow, Russia.} }
\date{
}

\maketitle
\thispagestyle{empty}

\begin{abstract}

It is shown that $A:= H_{1, \,\eta}(G)$, the sympectic reflection algebra,
 has $T_G$ independent traces, where $T_G$ is the
number of conjugacy classes of elements without eigenvalue $1$
belonging to the finite group $ {G} \subset Sp(2N) \subset End({\mathbb  C}^{2N})$
generated by the system of symplectic reflections.

Simultaneously, we show that the algebra
$ A$,  considered as a superalgebra with a natural
parity, has $S_G$ independent  supertraces, where $S_G$ is
the number of conjugacy classes of elements without eigenvalue $-1$
belonging to $ {G} $.

We consider also $A$ as a Lie algebra $A^L$ and as a Lie superalgebra $A^S$.

It is shown  that if $A$ is a simple associative algebra,
then the supercommutant $[A^{S},A^{S}]$ is a simple Lie superalgebra having
at least $S_G$ independent supersymmetric invariant non-degenerate  bilinear forms,
and
the quotient $[A^L,A^L]/([A^L,A^L]\cap\mathbb C)$ is a simple Lie algebra having
at least $T_G$ independent symmetric invariant non-degenerate  bilinear forms.
\end{abstract}




\section{Introduction}

In \cite{KT}, it was shown
that $H_{W({\mathfrak R})}(\eta)$, the algebra  of observables of
the rational Calogero model based on the root system ${\mathfrak R} \subset
{\mathbb  R}^N$, has $T_{\mathfrak R}$ independent traces, where $T_{\mathfrak R}$ is the
number of conjugacy classes of elements without eigenvalue $1$
belonging to the Coxeter group $W({\mathfrak R})\subset End({\mathbb  R}^N)$
generated by the root system ${\mathfrak R}$,
and that
the algebra
$H_{W({\mathfrak R})}(\eta)$,  considered as a superalgebra with a natural
parity, has $S_{\mathfrak R}$ independent  supertraces, where $S_{\mathfrak R}$ is
the number of conjugacy classes of elements without eigenvalue $-1$
belonging to $W({\mathfrak R})$.

Unlike the case of finite-dimensional associative algebras,
the presence of several (super)traces on the  infinite-dimensional
superalgebra $H_{W({\mathfrak R})}(\eta)$
in the case of irreducible ${\mathfrak R}$
does not necessarily imply
violation of simplicity
except certain particular values of parameter(s) $\eta$.

It is easy to show that $H_{W({\mathfrak R})}(\eta)=H_{1, \,\eta}(W({\mathfrak R}))$, where
$ H_{t, \,\eta}(G)$ is a symplectic reflection algebra introduced in \cite{sra}
for any finite group $G\subset Sp(2N)$
generated by symplectic reflections.
Here we extend the results of \cite{KT} from $H_{1, \,\eta}(W({\mathfrak R}))$
to $ H_{1, \,\eta}(G)$.

Besides, we consider Lie (super)algebras generated by $\HH$ and
invariant (super)symmetric bilinear forms on these Lie (super)algebras
generated by the traces and supertraces.

\section{Preliminaries}

\subsection{Traces}

Let ${\cal A}$ be an associative superalgebra with parity $\pi$.
All expressions of linear algebra are given for homogenous elements only
and are supposed to be extended to inhomogeneous elements via linearity.

A linear function $str$ on ${\cal A}$ is called a {\it supertrace} if
$$str(fg)=(-1)^{\pi(f)\pi(g)}str(gf) \ \mbox{ for all } f,g\in {\cal A}.$$

A linear function $tr$ on ${\cal A}$ is called a {\it trace} if
$$tr(fg)=tr(gf) \ \mbox{ for all } f,g\in {\cal A}.$$

Let $\varkappa=\pm 1$.
We can unify the definitions of trace and supertrace by introducing a $\varkappa$-trace.
We say that a linear function\footnote{From the German word {\it Spur}.}
$\str$
on ${\cal A}$ is a {\it $\varkappa$-trace} if
\begin{equation}\label{scom}
\str(fg)=\varkappa^{\pi(f)\pi(g)}\str(gf) \ \mbox{ for all $f,g\in {\cal A}$}.
\end{equation}

 A linear function $L$ is {\it even} (resp. {\it odd}) if $L(f)=0$
for any odd (resp. even) $f\in{\cal A}$.

Clearly, any linear function $L$
can be decomposed in the sum $L=L_++L_-$
of even linear function $L_+$ and odd linear function $L_-$.

Observe that each odd trace is simultaneously an odd supertrace and vice versa.

Let ${\cal A}_1$ and ${\cal A}_2$ be associative superalgebras with parities $\pi_1$ and $\pi_2\,$,
respectively.

Define their tensor product%
\footnote{In this paper we do not need {\it the supertensor product}
introduced by setting
\[
(a_1 \otimes a_2)(b_1 \otimes b_2)=(-1)^{\pi_1(b_1)\pi_2(a_2)}(a_1
b_1) \otimes (a_2 b_2)\text{~~for any $a_1, b_1\in {\cal A}_1,\ a_2,
b_2\in {\cal A}_2$}.\] }
${\cal A}={\cal A}_1\otimes {\cal A}_2$
as a superalgebra with the product
\be\nonumber
(a_1 \otimes a_2)(b_1 \otimes b_2)=(a_1 b_1) \otimes (a_2 b_2)
\text{~~for any
$a_1, b_1\in {\cal A}_1,\ a_2, b_2\in {\cal A}_2$}
\ee
and the
parity
$\pi$ defined by the formula
$\pi(a_1\otimes a_2)=\pi_1(a_1)+\pi_2(a_2)$.

Let $T_i$ be a trace on  ${\cal A}_i$.
Clearly, the function $T$ such that $T(a\otimes b)=T_1(a)T_2(b)$ is a trace on ${\cal A}$.

Let $S_i$ be an {\it even} supertrace on  ${\cal A}_i$.
Clearly, the function $S$ such that $S(a\otimes b)=S_1(a)S_2(b)$ is an even supertrace on ${\cal A}$.

In what follows, we use three types of brackets:
\begin{eqnarray}
[f,g]&=&fg-gf ,   \nonumber \\
\{f,g\}&=&fg+gf ,\nonumber \\
\lbrack f,g\rbrack _{\varkappa} &=& fg -\varkappa^{\pi(f)\pi(g)}gf.\nonumber
\end{eqnarray}

Every $\varkappa$-trace $\str (\cdot)$ on superalgebra ${\cal A}$ generates
the following bilinear form on ${\cal A}$:
\bee\label{bf}
B_{\str }(f,g):=\str (f\cdot g) \mbox{ for any } f,g\in {\cal A}.
\eee

It is obvious that if such a bilinear form $B_{\str }$ is degenerate,
then the
null-vectors of $B_{\str }$ (i.e., $v \in {\cal A}$ such that $B(v, x)=0$ for any
$x\in {\cal A}$)
constitute the two-sided
ideal ${\cal I}\subset {\cal A}$.
If the $\varkappa$-trace generating degenerate bilinear form is homogeneous (even or odd),
then the corresponding ideal is a superalgebra.

If $\varkappa=-1$, the ideals of this sort are present, for example,
in the superalgebras $H_{1,\eta}(A_{1})$
(corresponding to the two-particle Calogero model) at $\eta =k+ \frac{1}{2} $,
see \cite{V}, and in the superalgebras $H_{1,\eta}(A_{2})$
(corresponding to three-particle Calogero model)
at $\eta =k+ \frac{1}{2} $
and $\eta=k\pm\frac1 3$, see \cite{K2},
for every integer $k$. For all
other values of $\eta$ all supertraces on these superalgebras
generate non-degenerate bilinear forms (\ref{bf}).

The general case of $H_{1,\eta}(A_{n-1})$ for arbitrary $n$ is considered
in \cite{IL}. Theorem 5.8.1 of \cite{IL} states that
the associative algebra $H_{1,\eta}(A_{n-1})$ is not simple if and only if $\eta=\frac{q}{m}$, where $q,m$
are mutually prime integers such that $1<m\leqslant n$, and presents the structure of corresponding ideals.

The dimension of the space of supertraces on $H_{1,\,\eta}(A_{n-1})$ is the number
of partitions of $n\geqslant 1$ into the sum of different positive integers, see \cite{KV},
and the space of the traces on $H_{W(A_{n-1})}(\eta)$ is one-dimensional
for $n\geqslant 2$ due to Theorem \ref{main1}, see also \cite{stek}.

So, every algebra $H_{1,\eta}(A_{n-1})$ with $\eta\ne \frac{q}{m}$, where $q,m$
are mutually prime integers, $1<m\leqslant n$, and $n\ge 2$,
is an example
of simple superalgebra with several independent supertraces (see also \cite{K3}).

\begin{conjecture}
{\it
Each of the ideals of $H_{1,\eta}(A_{n-1})$ is
the set of null-vectors of the degenerate bilinear form (\ref{bf})
for some $\varkappa$-trace $sp$
on $H_{1,\eta}(A_{n-1})$.%
}
\end{conjecture}

More examples of associative simple (super)algebra with several (super)traces
are presented in Section \ref{non}.

\subsection{Symplectic reflection group}

Let $V={\mathbb  C}^{2N}$ be endowed with a non-degenerate
anti-symmetric
$Sp(2N)$-invariant bilinear form $\omega(\cdot,\cdot)$,
let the vectors $ a_i \in V$, where $i=1,\,...\,,\,2N$, constitute
a basis in $V$.

The matrix $(\omega_{ij}):=\omega(a_i,\,a_j)$ is anti-symmetric and non-degenerate.

Let $x^i$ be the coordinates of ${ x}\in V$, i.e.,  ${
x}= a_i\,x^i$. Then $\omega({ x},\, y)=
\omega_{ij} x^i y^j$ for
any ${ x},\, y \in V$. The indices $i$ are raised and lowered
by means of the forms $(\omega_{ij})$ and $(\omega^{ij})$, where $\omega_{ij}\omega^{kj}=\delta_i^k$.

\begin{definition}
The element $R\in Sp(2N)\subset End V$ is called a {\it symplectic reflection}, if
$rk(R-1)=2$.
\end{definition}

\begin{definition}
Any finite subgroup $G$ of $Sp(2N)$ generated by a set of symplectic reflections
is called a {\it symplectic reflection group}.
\end{definition}

We collect some elementary properties of the elements of the symplectic reflection group
in the following Proposition.

\begin{proposition}\label{collect}
{\it
Let $G$ be a symplectic reflection group and $g\in G$.
Then
\begin{enumerate}
\item \label{Jord}
\vskip -3mm
The Jordan normal form of $g$ is diagonal.
\item
Each eigenvalue of $g$ is a root of unity.
\item
$\mathrm{det}\, g=1$.
\item
If $\lambda $ is an eigenvalue of $g$, then $\lambda ^{-1}$
is also an eigenvalue of $g$.
\item \label{even}
The spectrum of $g$ has an even number of $-1$ and an even number of $+1$.
\item \label{6}
$g^{tr}\omega g=\omega $, where $g^{tr}$ is the transposed of $g$,
 or, equivalently,  $g^k_i \, \omega_{kl} \, g^l_j=\omega_{ij}$.
.
\end{enumerate}
}
\end{proposition}

\vskip -2mm
Clearly, each item of Proposition \ref{collect} follows either from the fact
that $G\subset Sp(2N)$ or from the fact that $G$ is a finite group.
Item \ref{6} is just the defining property of $Sp(2N)$.

In what follows, $G$ stands for a symplectic reflection group, and
$\R$ stands for the set of all symplectic reflections in $G$.

Let $R\in \R$. Set%
\footnote{Hereafter we denote all the units in groups,
algebras, etc, by 1, and $c\cdot 1$ by $c$ for any number $c$.}
\bee
V_R &:=& Im (R-1) \,, \\
Z_R &:=& Ker (R-1) \,.
\eee

Clearly, $V_R$ and $Z_R$ are symplectically perpendicular, i.e., $\omega(V_R,\,Z_R)=0$,
and  $V=V_R\oplus^\bot Z_R$.
Hereafter  the expression $U\oplus ^{\bot }W$
denotes a direct sum whose summands are symplectically
perpendicular to each other.

So, let $x=x_{{\phantom \,}_{V_R}}+x_{{\phantom \,}_{Z_R}}$ for any $x\in V$,
where $x_{{\phantom \,}_{V_R}}\in V_R$ and $x_{{\phantom \,}_{Z_R}}\in Z_R$.
Set
\be
\omega_R(x,y):=\omega (x_{{\phantom \,}_{ V_R}},\,y_{{\phantom \,}_{V_R}}).
\ee

Item \ref{even} of Proposition \ref{collect}
allows one to introduce the following grading on $\mathbb C[G]$.
Recall that $\varkappa = \pm 1$, and that we consider both values of $\varkappa$.

\begin{definition}
Let the grading $E$ on $\mathbb C[G]$ be defined  by the formula
\be
E(g) := \frac 1 2 \dim {\cal E}(g) \texttt{  for any $g\in G$},
\ee
 where
\be
{\cal E}(g) := Ker(g-\varkappa).
\ee
\end{definition}

For any $g\in  {G} $, the number $E(g)$ is an integer such that $0\leqslant E(g) \leqslant N$.

The following Lemma is crucial in what follows.%
\footnote{
An analogous Lemma is proved in \cite{KT} for the \emph{real orthogonal} matrices
and \emph{reflections} in $\mathbb R^N$.
}

\begin{lemma}\label{grad-1}
{\it
Let $g\in G$, $R\in \R$. If there exist
$c^1,\, c^2 \in Ker(g-\varkappa)$ such that $\omega_R(c^1,\,c^2)\ne 0$,
then
\be
E(Rg)=E(g)-1.
\ee
Besides,
\be
{\cal E}(Rg)=Z_R \cap {\cal E}(g).
\ee
}
\end{lemma}
\begin{proof}
Clearly, $Rgx=\varkappa Rx=\varkappa x$ if $gx=\varkappa x$ and $Rx=x$.
Hence, $Z_R \cap {\cal E}(g) \subset {\cal E}(Rg)$.

Denote
\be
{\cal E}_{ R}(g) := Z_{R}\cap {\cal E}(g).
\ee

Since $\omega_R(c^1,\,c^2)\ne 0$, it follows that the vectors
$c^l_{V_{R}}\in V_R$, where $l=1,2$, are independent,
so the $c^l_{V_{R}}$ constitute a
basis of $V_{R}$ (recall that $\dim V_R=2$).

Clearly, ${\cal E}_{ R}(g)\oplus span (c^1,\, c^2) = {\cal E}(g)$,
implying $\dim {\cal E}_{ R}(g)+2 = \dim {\cal E}(g)$.

It remains to prove that
${\cal E}_R(g) = {\cal E}(Rg)$

Suppose that there exists a vector $u\in {\cal E}(Rg)$ such that $u\notin {\cal E}_R(g)$.
Since $\dim {\cal E}(Rg)$ is even, our supposition implies that
there exist two vectors $u^1,\, u^2 \in {\cal E}(Rg)$ such that
$span (u^1, u^2) \cap {\cal E}_R(g)=0$.

Let $Z_{\mathrm{rem}}\subset Z_R$ be subspace of $Z_R$
such that $Z_R= Z_{\mathrm{rem}}\oplus {\cal E}_{ R}(g)$, so
$\dim Z_{\mathrm{rem}}=2s$, where $s:=N-E(g)$.
%
%
%
Then
\begin{eqnarray}
V &=&V_{R}\oplus ^{\bot }[Z_{\mathrm{rem}}\oplus {\cal E}_{ R}(g)].
\label{1.2.2.1}
\end{eqnarray}

The vectors $c^l$ ($l=1,2$) can be decomposed according to decomposition (\ref{1.2.2.1}):
\be
c^l=c^l_{V_R}+c^l_{\mathrm{rem}}+c^l_{{\cal E}_R(g)}.
\ee

Define a linear map $\rho: \ V_R\mapsto Z_{\mathrm{rem}}$ by the formula
\be
\rho c^l_{V_R}=c^l_{\mathrm{rem}}.
\ee
Clearly,
$x+\rho x\in {\cal E}(g)$ for each $x\in V_R$.

In the decomposition (\ref{1.2.2.1})
the matrices of $R$ and $g$ have the block forms%
\begin{equation}\label{3by3}
R=\left(
\begin{array}{ccc}
R_{2\times 2} & 0 & 0 \\
0 & \mathrm{1}_{2s\times 2s} & 0 \\
0 & 0 & \mathrm{1}_{(2E_{}(g)-2)\times (2E_{ }(g)-2)}%
\end{array}%
\right),\;g=\left(
\begin{array}{ccc}
g_{11} & g_{12}^{\prime } & 0 \\
g_{21}^{\prime } & g_{22}^{\prime } & 0 \\
g_{31}^{\prime } & g_{32}^{\prime } & \varkappa _{(2E_{
}(g)-2)\times (2E_{ }(g)-2)}%
\end{array}%
\right),
\end{equation}%
%
where the blocks of $g$ are of the same sizes as those of $R$.

From previous consideration we know that $g$ in (\ref{3by3}) has
2-dimensional space of eigenvectors $c$ with eigenvalue $\varkappa$, which
can be written in the form%
\begin{equation*}
c=\left(
\begin{array}{c}
x \\
\rho x \\
0%
\end{array}%
\right) , \mbox{ where }
\left(
\begin{array}{c}
x \\
0 \\
0%
\end{array}%
\right) \in V_{R},\;\left(
\begin{array}{c}
0 \\
\rho x \\
0%
\end{array}%
\right) \in Z_{\mathrm{rem}},
\end{equation*}%
i.e., the following relations take place
\begin{equation*}
g_{11}=\varkappa -g_{12}^{\prime }\rho ,\ \ \ g_{21}^{\prime }=(\varkappa
-g_{22}^{\prime })\rho ,\ \ \ g_{31}^{\prime }=-g_{32}^{\prime }\rho .
\end{equation*}

Let us look for null-vectors $u$ of $Rg-\varkappa $ on $V_{R}\oplus
^{\bot }Z_{\mathrm{rem}}$ in the form
\begin{equation*}
u=\left(
\begin{array}{c}
x \\
\rho x+z \\
0%
\end{array}%
\right) ,
\end{equation*}%
which is, in fact, a general form of such $u$.

Since
\begin{equation*}
Rg=\left(
\begin{array}{ccc}
R_{2\times 2}g_{11} & R_{2\times 2}g_{12}^{\prime } & 0 \\
g_{21}^{\prime } & g_{22}^{\prime } & 0 \\
\;g_{31}^{\prime } & g_{32}^{\prime } & \varkappa%
\end{array}%
\right) ,
\end{equation*}%
the equation%
\begin{equation}
(Rg-\varkappa )c_{R}=0  \label{1.2.2.2}
\end{equation}%
gives
\begin{eqnarray}
&&\varkappa (R_{2\times 2}-1)x+R_{2\times 2}g_{12}^{\prime }z=0,
\label{1.2.2.3a} \\
&&\,(g_{22}^{\prime }-\varkappa )z=0,\ \ \ g_{32}^{\prime }z=0.  \label{1.2.2.3b}
\end{eqnarray}

So, if eqs. (\ref{1.2.2.3b}) do not have any nontrivial solutions, eq. (\ref%
{1.2.2.2}) has no nontrivial solutions either. If eqs. (\ref{1.2.2.3b})
have null-vectors $z_{0}$ such that $g_{12}^{\prime }z_{0}=0$, then eq. (\ref%
{1.2.2.3a}) shows that\ $x=0$, and we see that $\left(
\begin{array}{c}
0 \\
z_{0} \\
0%
\end{array}%
\right) \in Z_{\mathrm{rem}}\cap {\cal E}_{ R}(g)$, which is impossible.

So, the only opportunity for $Rg$ to have eigenvalue $\varkappa $ with
multiplicity $> 2E(g)-2$ (see eq. (\ref {3by3})),  is the
existence of a vector %
\begin{equation*}
u=\left(
\begin{array}{c}
0 \\
z \\
0%
\end{array}%
\right)
\in Z_{\mathrm{rem}}
\end{equation*}%
which satisfies eqs. (\ref{1.2.2.3b}) and $g_{12}^{\prime }z\neq 0$, i.e.,
\begin{equation*}
(g-\varkappa )u=\left(
\begin{array}{c}
g_{12}^{\prime }z \\
(g_{22}^{\prime }-\varkappa )z \\
g_{32}^{\prime }z%
\end{array}%
\right) =\left(
\begin{array}{c}
g_{12}^{\prime }z \\
0 \\
0%
\end{array}%
\right) \in V_{R}.
\end{equation*}%
Because the multiplicity of $\varkappa $ in the spectrum of $Rg$ is even, the
supposition that $Rg-\varkappa$ has null-vectors
besides ${\cal E}_R(g)$
leads to existence of a
2-dimensional subspace $Z_{0}\subset Z_{\mathrm{rem}}$ such that%
\begin{equation}
(g-\varkappa )Z_{0}=V_{R}.  \label{1.2.2.4}
\end{equation}

\textbf{Suppose that $Z_{0}\neq 0$}.
Represent $Z_{\mathrm{rem}}$ in the form
$Z_{\mathrm{rem}}=Z_{0}\oplus Z_{\mathrm{r}}$.
In the basis for decomposition
$V=V_{R.}\oplus ^{\bot }[(Z_{0}\oplus Z_{\mathrm{r}}\oplus {\cal E}_{ R}(g)]$,
the matrix $g$ has the form%
\begin{equation*}
g=\left(
\begin{array}{cccc}
g_{11} & g_{12} & g_{13} & 0 \\
g_{21} & g_{22} & g_{23} & 0 \\
g_{31} & g_{32} & g_{33} & 0 \\
g_{41} & g_{42} & g_{43} & \varkappa%
\end{array}%
\right) ,
\end{equation*}%
where%
\begin{eqnarray*}
\left(
\begin{array}{c}
g_{21} \\
g_{31}%
\end{array}%
\right) &=&g_{21}^{\prime },\;g_{41}=g_{31}^{\prime },\;\left(
\begin{array}{cc}
g_{12} & g_{13}%
\end{array}%
\right) =g_{12}^{\prime }, \\
\left(
\begin{array}{cc}
g_{22} & g_{23} \\
g_{32} & g_{33}%
\end{array}%
\right) &=&g_{22}^{\prime },\;\left(
\begin{array}{cc}
g_{42} & g_{43}%
\end{array}%
\right) =g_{32}^{\prime },
\end{eqnarray*}%
\begin{eqnarray*}
g_{11} &=&\varkappa -g_{12}\rho _{1}-g_{13}\rho _{2}, \\
g_{21} &=&(\varkappa -g_{22})\rho _{1}-g_{23}\rho _{2},\;g_{31}=(\varkappa
-g_{33})\rho _{2}-g_{32}\rho _{1}, \\
\;g_{41} &=&-g_{42}\rho _{1}-g_{43}\rho _{2},
\end{eqnarray*}%
and where $\rho _{1}$ and $\rho _{2}$ give the decomposition of $\rho $:%
\begin{equation*}
\rho x=\rho _{1}x+\rho _{2}x, \text{ where } \rho _{1}x\in Z_{0}\text{ , \ }%
\rho _{2}x\in Z_{\mathrm{r}}.
\end{equation*}%
Due to condition (\ref{1.2.2.4}), the matrix $g$ acquires the form%
\begin{equation}
g=\left(
\begin{array}{cccc}
g_{11} & g_{12} & g_{13} & 0 \\
g_{21} & \varkappa & g_{23} & 0 \\
g_{31} & 0 & g_{33} & 0 \\
g_{41} & 0 & g_{43} & \varkappa%
\end{array}%
\right) ,\;\det g_{12}\neq 0 ,  \label{1.2.2.5}
\end{equation}%
and the symplectic form $\omega $ has the
shape
$\omega =\left(
\begin{array}{cccc}
\omega _{\text{2}\times \text{2}}^{R} & 0 & 0 & 0 \\
0 & \omega _{\text{2}\times \text{2}}^{22} & \omega ^{23} & \omega ^{24} \\
0 & \omega ^{32} & \omega ^{33} & \omega ^{34} \\
0 & \omega ^{42} & \omega ^{43} & \omega ^{44}%
\end{array}%
\right)$, where $\omega _{\text{2}\times \text{2}}^{R}$ is non-degenerate.
 Due to (\ref{1.2.2.5}), the equality
 $\omega =g^{\mathrm{tr}}\omega g$
gives for the $22$-block:%
\begin{eqnarray*}
&&\omega _{\text{2}\times \text{2}}^{22}=\left. \left(
\begin{array}{cccc}
\ast & \ast & \ast & \ast \\
g_{12}^{\mathrm{tr}} & \varkappa & 0 & 0 \\
\ast & \ast & \ast & \ast \\
0 & 0 & 0 & \varkappa%
\end{array}%
\right) \left(
\begin{array}{cccc}
\omega _{\text{2}\times \text{2}}^{R} & 0 & 0 & 0 \\
0 & \omega _{\text{2}\times \text{2}}^{22} & \ast & \ast \\
0 & \ast & \ast & \ast \\
0 & \ast & \ast & \ast%
\end{array}%
\right) \left(
\begin{array}{cccc}
\ast & g_{12} & \ast & 0 \\
\ast & \varkappa & \ast & 0 \\
\ast & 0 & \ast & 0 \\
\ast & 0 & \ast & \varkappa%
\end{array}%
\right) \right| _{22}= \\
&&\,=\left. \left(
\begin{array}{cccc}
\ast & \ast & \ast & \ast \\
g_{12}^{\mathrm{tr}} & \varkappa & 0 & 0 \\
\ast & \ast & \ast & \ast \\
0 & 0 & 0 & \varkappa%
\end{array}%
\right) \left(
\begin{array}{cccc}
\ast & \omega _{\text{2}\times \text{2}}^{R}g_{12} & \ast & \ast \\
\ast & \varkappa \omega _{\text{2}\times \text{2}}^{22} & \ast & \ast \\
\ast & \ast & \ast & \ast \\
\ast & \ast & \ast & \ast%
\end{array}%
\right) \right| _{22}=g_{12}^{\mathrm{tr}}\omega _{\text{2}\times \text{2}%
}^{R}g_{12}+\omega _{\text{2}\times \text{2}}^{22}
\end{eqnarray*}%
or $g_{12}^{tr}\omega _{\text{2}\times \text{2}}^{R}g_{12}=0$ which
contradicts the nondegeneracy of $g_{12}$.

So, $Z_0=0$ and
the matrix $Rg-\varkappa$ has no null-vectors besides ${\cal E}_R(g)$.
\end{proof}

\section{Symplectic reflection algebra}
The superalgebra $ H_{1, \,\eta}(G)$
is a deform of the skew product%
\footnote{
Let ${\cal A}$ and ${\cal B}$ be superalgebras, and ${\cal A}$ a ${\cal B}$-module.
We say that the superalgebra
${\cal A}\ast {\cal B}$ is a {\it skew product} of ${\cal A}$ and ${\cal B}$ if
${\cal A}\ast {\cal B}={\cal A} \otimes {\cal B}$ as a superspace and
$(a_1\otimes b_1)\ast (a_2\otimes b_2)=a_1 b_1(a_2)\otimes b_1 b_2$.
The element $b_1(a_2)$ may include a sign factor imposed by the Sign Rule.
}
 of the Weyl algebra $W_N$ and the group algebra
of a finite subgroup $G \subset Sp(2N)$
generated by symplectic reflections, see
Definition \ref{defpage} below.

\subsection{Definitions}

Let $ {\mathbb C}[G] $ be the {\it group algebra} of $ {G} $, i.e., the
set of all linear combinations $\sum_{g\in  {G} } \alpha_g \bar g$,
where $\alpha_g \in {\mathbb C}$, and we temporarily write $\bar g$
to distinguish $g$ considered as an element of $ {G} \subset End(V)$
from the same element $\bar g \in  {\G} $
considered as an element of the group algebra.
The addition in $ {\G} $ is defined as follows:
$$
\sum_{g\in  {G} } \alpha_g \bar g + \sum_{g\in  {G} } \beta_g \bar g
= \sum_{g\in  {G} } (\alpha_g + \beta_g) \bar g
$$
and the multiplication is defined by setting
$\overline {g_1\!}\,\, \overline {g_2\!} = \overline {g_1 g_2}$.

Let $\eta$ be a function on ${\R}$, i.e., a set of constants $\eta_R$ with $R\in\R$ such that
$\eta_{R_1}=\eta_{R_2}$ if $R_1$ and $R_2$ belong to
one conjugacy class of $ {G} $.

\begin{definition}
\label{defpage}
The algebra $H_{t,\eta}(G)$, where $t\in \mathbb C$ is an associative algebra with unity {\bf 1};
it is the algebra
${\mathbb C}[V]$ of polynomials in the elements of $V$ with coefficients
in the group algebra ${\mathbb C}[G]$ subject to the relations
\bee
g x&=&g(x) g %
\mbox{ for any } g\in  {G}
                   \mbox{ and } x \in V,
                   \mbox{ where } g(x)= a_i g^i_j x^j
                   \mbox{ for }x=a_ix^i,\\
\label{rel}
 \!\!\!\!\!\!\!\!\!\!\!\! [ x  , y] &=& t \omega(x,y)
 +
       \sum_{R\in\R} \eta_R
\omega_R(x,y)R
\mbox{ for any  $x,y\in V$}.
\eee

The algebra $H_{t,\eta}(G)$ is called a {\it symplectic reflection algebra},
see \cite{sra}.
\end{definition}

The commutation relations (\ref{rel}) suggest
to define the {\it parity} $\pi$ by setting:
\be
\pi (x)=1,\ \pi (g)=0
\ \mbox{ for any }x\in V, \mbox{ and }g\in G,
\ee
enabling one to consider $\HH$ as an associative {\it superalgebra}.

We consider the case $t\ne 0$ only, which is equivalent to the case $t=1$.

\subsection{Bases of eigenvectors}

We say that a polynomial $f\in {\mathbb C}[V]$ is {\it monomial}
if it can be expressed in the form $f=u_1u_2...u_k$, where
$u_i\in V$.

We say that an element $h\in \HH$  is {\it monomial}
if it can be expressed in the form $h=u_1u_2...u_k g$, where
$u_i\in V$ and $g\in G$.

Due to item \ref{Jord} of Proposition \ref{collect}, for each $g\in G$,
there exists a basis ${\mathfrak B}_g=\{b_1,\,...\,,\,b_{2N}\}$ of $V$ such that
(no summation here)
\be\label{basis}
g(b_I)=\lambda_I b_I, \mbox{ where } I=1,2,...,2N,
\ee
or, equivalently,
\be\label{basis1}
g b_I=\lambda_I b_I g.
\ee

We can represent any element $h\in \HH$ in the form $h=\sum_{g\in G}h_g g$,
where the polynomials $h_g$ depend on $b_I\in {\mathfrak B}_g$.

\begin{definition}\label{reg-spec} Let $b_I\in {\mathfrak B}_g$.
A monomial $b_{I_1}\dots b_{I_k}g$ is said to be {\it regular}
if $\lambda_{I_s}\ne \varkappa$
for some $s$, where $1\leq s \leq k$, and
{\it special} if $\lambda_{I_s} = \varkappa$
for each $s$, where $1\leq s \leq k$.
\end{definition}

Set also:
\bee
F_{IJ}&:=& [b_I,\, b_J], \\
\label{calc}
{\cal C}_{IJ} &:=& \omega(b_I,\, b_J),\\
f_{IJ} &:= & [b_I,\, b_J]- {\cal C}_{IJ}.
\label{37}
\eee

\begin{lemma}\label{lemma4}
{\it
Let $g\in  {G}$.
Let
$b_I, b_J\in {\cal E} (g)$.
Then
\be\label{main}
E(f_{IJ} \sigmama )=E(\sigmama )-1
 .
\ee
}
\end{lemma}

\indent\begin{proof}
Proof follows from eq. (\ref{rel}) (recall that $t=1$) and Lemma \ref{grad-1}.
\end{proof}

\subsection{Partial orderings in $H_{1,\eta}(G)$}

\begin{definition}\label{ordering}
Let $f_1, f_2 \in {\mathbb C}[V]$ be monomials
either both even or both odd.
Let $g_1, g_2 \in G$.
We say $f_1 g_1 < f_2 g_2$ if either
 $deg f_1<deg f_2$ or $deg f_1 = deg f_2$ and $E(g_1) < E(g_2)$.
\end{definition}

It is easy to describe all {\it minimal elements} in $H_{1,\eta}(G)$,
i.e., the elements $f_{min}$ such that there exists $f\in H_{1,\eta}(G)$
such that $f_{min} < f$, and
there are no  elements $f_<$ such that $f_< < f_{min}$:

a) In the even subspace of $H_{1,\eta}(G)$ the minimal elements are
$g\in G$ such that $E(g)=0$.

b) In the odd subspace of $H_{1,\eta}(G)$, the minimal elements are
the elements of the form $xg$, where $x\in V$, $g\in G$ and $E(g)=0$.

\begin{proposition}\label{oddmin}
{
\it
Let $\varkappa=1$. Then for each trace $tr$ and for each odd minimal
element $xg$, the following equality takes place
\be\label{oddminzero}
tr(xg)=0.
\ee
}
\end{proposition}

\begin{proof}
Since $\varkappa=1$ and $E(g)=0$, the element $g$ does not have
eigenvalue $+1$.

Decompose $x\in V$ in the basis ${\mathfrak B}_g$:
$x=b_I x^I $, where $g(b_I)=\lambda_I b_I$ and $\lambda_I\ne 1$ for
any $I=1,\dots ,2N$.

Further, $tr(b_I g)=tr(gb_I)=tr(g(b_I) g)=\lambda_I tr(b_I g)$, which implies
$tr(b_I g)=0$ and, as a consequence, $tr(xg)=0$.
\end{proof}

Consider the defining relations (\ref{scom})
as a system of linear equations for the linear function $\str$.
Clearly, this system is equivalent to the following two its subsystems:
\bee\label{ma1}
&& \str \left( [b_I , P (a) \sigmama
]_\varkappa\right)=0,\\
\label{ma2}
&&  \str \left(  \tau^{-1} P (a)\sigmama\tau
\right)=
\str \left( P (a) \sigmama
\right)
\eee
for all monomials
$P\in {\mathbb C}[V]$, $b_I\in {\mathfrak B}_g$, and $g,\tau\in  {G} $.

If the $\varkappa$-trace is  either even or the $\varkappa$-trace is odd and $\varkappa=1$, then
eq. (\ref{ma1}) can be rewritten in the form
\be\label{evenform}
\str \left( b_I  P (a) \sigmama-
\varkappa P (a) \sigmama b_I \right)=0.
\ee

Eq. (\ref{evenform}) enables us to express a $\varkappa$-trace of any even monomial in $\HH$
in terms of the $\varkappa$-trace of even minimal elements. Besides, it implies that each
odd trace on $\HH$ is equal to zero.
Both these statements can be proved in a finite number of the following step operations.

{\bf Regular step operation.}
Let $b_{I_1}b_{I_2}\,\dots\,b_{I_k}g$ be a regular monomial.
Up to a polynomial of lesser degree, this monomial can be expressed in a form
such that $\lambda_{I_1}\ne \varkappa$.

Then
$$
\str(b_{I_1}b_{I_2}\,\dots\,b_{I_k}g)=\varkappa \str(b_{I_2}\,\dots\,b_{I_k}gb_{I_1})=
\varkappa \lambda_{I_1}\str(b_{I_2}\,\dots\,b_{I_k}b_{I_1}g),
$$
which implies
$$
\str(b_{I_1}b_{I_2}\,\dots\,b_{I_k}g)-
\varkappa \lambda_{I_1}\str(b_{I_1}b_{I_2}\,\dots\,b_{I_k}g)
=\varkappa \lambda_{I_1}\str([b_{I_2}\,\dots\,b_{I_k}, \, b_{I_1}]\,g).
$$
Thus,
\be\label{lessdeg}
\str(b_{I_1}b_{I_2}\,\dots\,b_{I_k}g)=
 \frac   {\varkappa \lambda_{I_1}} {1-\varkappa \lambda_{I_1}} \str([b_{I_2}\,\dots\,b_{I_k}, \, b_{I_1}]\,g).
\ee

This step operation expresses the $\varkappa$-trace of any regular
degree $k$ monomial in terms of the $\varkappa$-trace of degree
$k-2$ polynomials.

{\bf Special step operation.}
Let $M\defeq b_{I_1}b_{I_2}\,\dots\,b_{I_k}g$ be a special monomial and $E(g)=l>0$.
The monomial $M$ can be expressed in the form
$$
M=b_{I}^p b_{J}^q \,b_{L_1}\,\dots\,b_{L_{k-p-q}}g +
\mbox{ a lesser-degree-polynomial},
$$
where
\bee
&& 0\leqslant p,q \leqslant k,\quad p+q\leqslant k,\nn
&& \lambda_I=\lambda_J=\lambda_{L_s}=\varkappa \quad \mbox{for } s=1,\,..., k-p-q, \\
&& {\cal C}_{IJ}=1,\quad {\cal C}_{IL_s}=0,\quad {\cal C}_{JL_s}=0\quad   \mbox{for } s=1,\,..., k-p-q
\,\,. \nonumber
\eee
Let $M'\defeq b_{I}^p b_{J}^q \,b_{L_1}\,\dots\,b_{L_{k-p-q}}\,$ and derive the equation for $\str (M'g)$. Since
$$
\str (b_J b_I M'g)=\varkappa \str(b_I M'gb_J)=
\str(b_I M'b_Jg),
$$
it follows that
\be\label{so}
\str([b_I M',\,b_J]g)=0.
\ee

Since $[b_I M',\,b_J]$ can be expressed in the form:
\bee
[b_{I}^{p+1} b_{J}^q \,b_{L_1}\,\dots\,b_{L_{k-p-q}},\, b_J]
&=&
\sum_{t=0}^p b_I^{\,t}(1+f_{IJ})b_I^{p-t}b_{J}^q \,b_{L_1}\,\dots\,b_{L_{k-p-q}}
\!\!
+
\nn
&+&
\!\! \!\!
\sum_{t=1}^{k-p-q}b_{I}^{p+1} b_{J}^q \,b_{L_1}\,\dots\, b_{L_{t-1}}\, f_{L_t\, J}
\, b_{L_{t+1}}\dots\,  b_{L_{k-p-q}}  \,,
\eee
it follows that eq. (\ref{so}) can be rewritten in the form
\bee\label{000}
(p+1)\str(M'g)= &-& \str \left(
\sum_{t=0}^p b_I^tf_{IJ}b_I^{p-t}b_{J}^q \,b_{L_1}\,\dots\,b_{L_{k-p-q}}g  +\right. \nn
&+&
\left. \sum_{t=1}^{k-p-q}b_{I}^{p+1} b_{J}^q \,b_{L_1}\,\dots\, b_{L_{t-1}}\, f_{L_t\, J}
\, b_{L_{t+1}}\dots\,  b_{L_{k-p-q}}g \right).
\eee

Due to Lemma \ref{grad-1} it is easy to see that eq. (\ref{000}) can be rewritten in the form
\be\label{step2}
\str(M'g)=\sum_{\tilde g\in  {G} :\, E(\tilde g)=E(g)-1} \str(P_{\tilde g} (a_{i})\tilde g),
\ee
where the $P_{\tilde g}$ are some polynomials such that $\deg P_{\tilde g} = \deg M'$.

So, the special step operation expresses the $\varkappa$-trace of a special
polynomial in terms
of the $\varkappa$-trace of polynomials  lesser in the sense of the ordering
introduced by Definition
\ref{ordering}.

Thus, we showed that it is possible to express the $\varkappa$-trace of any polynomial
as a linear combination of the $\varkappa$-trace of minimal elements of $\HH$
using a finite number of regular and special step operations.

Since each step operation is manifestly $ {G} $-invariant,
the resulting $\varkappa$-trace is also $ {G} $-invariant
if the $\varkappa$-trace
of minimal elements of $\HH$ is
$ {G} $-invariant.

Due to Proposition \ref{oddmin}, each trace of any odd minimal element is zero,
so each odd trace is zero.
But since each odd trace is also a supertrace, we can say
that each odd $\varkappa$-trace is zero.

These arguments proved the following Theorem and Proposition:

\begin{theorem}\label{even1}
{\it
Each nonzero $\varkappa$-trace on $H_{1,\,\eta}(G)$ is even.
}\end{theorem}

\begin{proposition}\label{tomin}
{\it
Each $\varkappa$-trace on $H_{1,\,\eta}(G)$ is completely defined by its values on
the minimal
elements of $G$.
}
\end{proposition}

Note that, due to $ {G} $-invariance,
the restriction of the $\varkappa$-trace on $G$ is a
{\it central function},
i.e.,
a function constant on the conjugacy classes.

Below we will prove that any central function on the set of minimal elements of $G$
can be extended to a $\varkappa$-trace on $\HH$.

\section{Ground Level Conditions}\label{anal1}
Clearly, ${\mathbb C}[G] $
is a  subalgebra of $\HH$.

It is easy to describe all $\varkappa$-traces on
$ {\mathbb C}[G] $. Every $\varkappa$-trace on $ {\mathbb C}[G] $
is completely determined by its values on
$ {G}
$
and  is a central function on $ {G} $
due to $ {G} $-invariance.
Thus, the number of $\varkappa$-traces on $\G$
is equal to the number of conjugacy classes in $ {G} $.

Since $\G \subset  \HH$, some
additional restrictions on these functions
follow from the definition (\ref{scom}) of $\varkappa$-trace
and the defining relations (\ref{rel}) for
$\HH$.  Namely, for any $g\in { {G} }$,
consider elements $c_I, c_J\in {\cal E}(g)$ such that
\be\label{eigss}
\sigmama c_I=\varkappa c_I \sigmama,\ \ \sigmama c_J=\varkappa c_J \sigmama.
\ee
Then, eqs. (\ref{scom}) and (\ref{eigss}) imply that
$$
\str    \left ( c_I c_J \sigmama
\right )=  \varkappa  \str    \left ( c_J \sigmama c_I\right )=  \str    \left ( c_J c_I
\sigmama \right ),
$$
and therefore
\be\label{mm}
\str    \left ( [ c_I, c_J] \sigmama \right )=0.
\ee

Since $[ c_I, c_J] \sigmama  \in \G $,
the conditions (\ref{mm}) single out the central functions on $\G$, which can
in principle be extended to $\varkappa$-traces on $\HH$,
and Theorem \ref{th6} states that each central function on $\G$
satisfying conditions (\ref{mm}) can indeed be
extended to a $\varkappa$-trace on $\HH$.
In \cite{KV}, the conditions (\ref{mm})  are called {\it Ground Level Conditions}.

\subsection{The solutions of Ground Level Conditions}

Ground Level Conditions (\ref{mm}) is an overdetermined system
of linear equations for the central functions on $\G$.

\begin{theorem}\label{th5}
{\it The dimension of the space of solutions of
Ground Level Conditions (\ref{mm}) is equal to the
number of conjugacy classes in $ {G} $ with $E(g)=0$.
Each central function on conjugacy classes in $ {G} $ with $E(g)=0$
can be uniquely extended to a solution of
Ground Level Conditions.
}
\end{theorem}

\subsection{Proof of Theorem \ref{th5}}
\label{anal3}

Let us prove a couple of simple statements we will use below.

\begin{proposition}\label{ah-ah0}
{\it
Let $h\in G$, $c\in {\cal E}(h)$,
$x\in {\mathfrak B}_h$, $h(x)=\lambda x$, where $\lambda \ne \varkappa$.
Then, for any central function $f$ on ${\mathbb C}[G]$, we have
\be
f ([c, x]h)\equiv 0.
\ee
}
\end{proposition}

\noindent
\begin{proof}Since $f$ is a central function, we have
$f ([c, x]h)=f (h[c, x]hh^{-1})=
f([h(c),h(x)]h)=\varkappa \lambda f ([c, x]h)$.
\end{proof}

\begin{proposition}\label{ah-ah}
{\it
Let $h\in G$, $c\in {\cal E}(h)$,
and Ground Level Conditions (\ref{mm}) be satisfied.
Then
\be\label{ah-ah1}
\str ([c, x]h)\equiv 0
\text{ for any } x\in V.
\ee
}
\end{proposition}
\begin{proof} Let $x=\sum_{\lambda\ne \varkappa}x_\lambda +x_\varkappa$,
where $h(x_\lambda)=\lambda x_\lambda$.
Since $\str $ is a central function,
Proposition \ref{ah-ah0} gives
$\str ([c, x]h)=\str ([c, x_\varkappa]h)$,
and eq. (\ref{mm}) gives $\str ([c, x_\varkappa]h)\equiv 0$.
\end{proof}

We prove Theorem \ref{th5} by induction on $E(g)$.

The first step is simple: if $E(g)=0$, then $\str (g)$ is an arbitrary
central function.
The next step is also simple: if $E(g)=1$, then there exists
a pair of elements $c_1,c_2\in {\cal E}(g)$
such that $\omega(c_1, c_2)\ne 0$.
Since $([c_1,\, c_2]-\omega(c_1, c_2))g\in \G$ and
$E(([c_1,\, c_2]-\omega(c_1,c_2))g)=0$ due to Lemma \ref{grad-1}, then
\be\label{sol}
\str (g)=-\frac 1 {\omega(c_1, c_2)} \str (([c_1,\, c_2]-\omega(c_1, c_2))g)
\ee
is the only possible value of $\str (g)$ for any $g\in G$ with $E(g)=1$.
Clearly, the right-hand side of eq. (\ref{sol}) does not depend on the
choice of basis vectors $c_1,c_2$ in ${\cal E}(g)$.

Suppose that the Ground Level Conditions (\ref{mm})
considered for all $g$ with $E(g)\leqslant l$
and for all $c_I, \, c_J \in {\cal E}(g)$
have $Q_l$ independent solutions.

\begin{proposition} \label{state}
{\it The value $Q_l$ does not depend on $l$.}
\end{proposition}

\indent\begin{proof}
It was shown above that $Q_1=Q_0$. Let $l\geqslant 1$.

Suppose that $Q_k$ does not depend on $k$ for $k\leqslant l$.
Consider $g\in  {G} $ with $E(g)=l+1$.
Let $c_I\in {\cal E}(g)$, where $I=1,2,...,2E(g) $, be
a basis in ${\cal E}(g)$ such that the symplectic form
${\cal C}_{IJ}=\omega(c_I,c_J)$
has a normal shape:
\begin{equation*}
{\cal C} =\left(
\begin{array}{ccccccc}
0 & 1 & 0 & 0 & \dots & 0 & 0 \\
-1 & 0 & 0 & 0 & \dots & 0 & 0 \\
0 & 0 & 0 & 1 & \dots & 0 & 0 \\
0 & 0 & -1 & 0 & \dots & 0 & 0 \\
\vdots & \vdots & \vdots & \vdots & \ddots & \vdots & \vdots \\
0 & 0 & 0 & 0 & \dots & 0 & 1 \\
0 & 0 & 0 & 0 & \dots & -1 & 0%
\end{array}%
\right).
\end{equation*}

We will show that for a fixed $g\in G$, all the Ground Level Conditions
\be\label{all}
{\cal C}_{IJ} \str (g) = -\str (([c_I,\, c_J]-{\cal C}_{IJ})g) \ \mbox{ for }\
I,J=1,\,\dots\,,\,2E(g)
\ee
follow from the  inductive hypothesis and just one of them, e.g.,
\be
\str (g) = -\str (([c_1,\, c_2]- 1)g).
\ee

For this purpose, it clearly suffices to consider eq. (\ref{all})
only for $I,J=1,...,4$:
\bee
\str (g) & = & -\str (([c_1,\, c_2]- 1)g),\label{12}\\
\str (g) & = & -\str (([c_3,\, c_4]- 1)g),\label{34}\\
0 & = & -\str ([c_1,\, c_3]g),\label{13}\\
0 & = & -\str ([c_1,\, c_4]g),\label{14}\\
0 & = & -\str ([c_2,\, c_3]g),\label{23}\\
0 & = & -\str ([c_2,\, c_4]g).\label{24}
\eee

Below we prove that eqs. (\ref{12}) and (\ref{34}) are equivalent,
namely
\be\label{eq}
\str (([c_1,\, c_2]- 1)g) \equiv \str (([c_3,\, c_4]- 1)g)
\ee
and eqs. (\ref{13}) -- (\ref{24}) follow from eq. (\ref{eq}).

Note that due to the inductive hypothesis
both sides of eq. (\ref{eq}) are well defined,
because
$$
E(([c_1,\, c_2]- 1)g)=E(([c_3,\, c_4]- 1)g)=l.
$$

Represent the left-hand side of eq. (\ref{eq}) as follows:
\bee
\str (([c_1,\, c_2]- 1)g) &=&\str ( A_{12}) + \str ( B_{12}), \text{ where }\\
A_{12} &:=& \sum_{R\in \R: \ \omega_R(c_3,\, c_4)=0}\eta_R \omega_R(c_1,\, c_2) Rg,
\label{A12}
\\
B_{12} &:=& \sum_{R\in \R: \ \omega_R(c_3,\, c_4)\ne 0}\eta_R \omega_R(c_1,\, c_2) Rg.
\label{B12}
\eee

Analogously,
\bee
\str (([c_3,\, c_4]- 1)g) &=& \str ( A_{34})+\str ( B_{34}), \text{ where }\\
A_{34} &:=& \sum_{R\in \R: \ \omega_R(c_1,\, c_2)=0}\eta_R \omega_R(c_3,\, c_4) Rg,
\label{A34}
\\
B_{34} &:=& \sum_{R\in \R: \ \omega_R(c_1,\, c_2)\ne 0}\eta_R \omega_R(c_3,\, c_4) Rg.
\label{B34}
\eee

It is clear from eqs. (\ref{A12}) -- (\ref{B34}) and Lemma \ref{grad-1} that
$$
E(A_{12})=E(B_{12})=E(A_{34})=E(B_{34})=l.
$$

Consider $R\in \R$ such that $\omega_R(c_1,c_2)\ne 0$.
Then there exists a
$2\times 2$ matrix $(U^R_{\alpha i})$, where $\alpha = 3,4$ and $i=1,2$,
such that
\bee\label{trans1}
 c_3^R &:=&c_3 - U^R_{31} c_1- U^R_{32} c_2 \in Z_R,
\\\label{trans2}
 c_4^R &:=&c_4 - U^R_{41} c_1- U^R_{42} c_2 \in Z_R.
\eee
This matrix defines the decomposition of ${c_3}_{V_R},{c_4}_{V_R}$
 in $V_R$ with respect to the basis ${c_1}_{V_R}, {c_2}_{V_R}$.

Clearly, due to Lemma \ref{grad-1} we have
\be
c_3^R,\, c_4^R\in {\cal E}(Rg).
\ee

If $\omega_R(c_3,c_4)\ne 0$, then $\det U^R\ne 0$.

If $\omega_R(c_3,c_4) = 0$, then $\det U^R = 0$.
Since
$$
\omega(c_3^R,\, c_4^R)=\omega(c_3,\,c_4)+\det U^R\omega(c_1,\,c_2)=1+\det U^R,
$$
it follows that if $\det U^R=0$, then $\omega(c_3^R,\, c_4^R)=1$.
So, if $\omega_R(c_3,c_4) = 0$ and $\omega_R(c_1,c_2) \ne 0$, then
\be\label{Rg}
\str (Rg)=-\str (([c_3^R,\, c_4^R]-1)Rg),
\ee
and
\be
E (([c_3^R,\, c_4^R]-1)Rg)=l-1.
\ee

Now, let us express $\str (A_{12})$
by means of the $\varkappa$-trace of elements of $G$ with grading $l-1$:
\bee\label{72}
\str (A_{12})&=&\sum_{R\in \R: \ \omega_R(c_3,\, c_4)=0}\eta_R \omega_R(c_1,\, c_2) \str (Rg)
=\nn
&=& - \sum_{R\in \R: \ \omega_R(c_3,\, c_4)=0}\eta_R \omega_R(c_1,\, c_2)
\str (([c_3^R,\, c_4^R]-1) Rg).
\eee

Since
\be
\str ([c_3^R, x]Rg) \equiv \str ([c_4^R, x]Rg) \equiv 0 \text{ for any } x\in V
\ee
due to Proposition \ref{ah-ah},
and since $\det U^R=0$ for any summand in eq. (\ref{72}),
we have
\be
\str (([c_3^R,\, c_4^R]-1) Rg)=\str (([c_3,\, c_4]-1) Rg),
\ee
and as a result
\bee\label{74}
\str (A_{12})= - \sum_{R\in \R: \ \omega_R(c_3,\, c_4)=0}\eta_R \omega_R(c_1,\, c_2)
\str (([c_3,\, c_4]-1) Rg)=
-\str (([c_3,\, c_4]-1)A_{12}).
\eee

As
$A_{12}=([c_1,\,c_2]-1)g - B_{12}$,
eq. (\ref{74}) gives
\bee
\str (A_{12}+B_{12})&=& \str (-([c_3,\, c_4]-1)(([c_1,\,c_2]-1)g - B_{12})+B_{12})=\nn
&=&
\str (-([c_3,\, c_4]-1)([c_1,\,c_2]-1)g + [c_3,\, c_4] B_{12}).
\label{1234}
\eee

Analogously,
\be\label{3412}
\str (A_{34}+B_{34})
=
\str (-([c_1,\, c_2]-1)([c_3,\,c_4]-1)g + [c_1,\, c_2] B_{34}).
\ee

\begin{proposition}\label{???}
{\it
Let $g\in G$, $c_1,c_2 \in {\cal E}(g)$, $\omega(c_1,c_2)=1$, and
$f$ a central function on ${\mathbb C}[G]$. Then
\be
f(([c_1,\, c_2]-1)([c_3,\,c_4]-1)g) = f(([c_3,\, c_4]-1)([c_1,\,c_2]-1)g).
\ee
}
\end{proposition}
\begin{proof}
Since $[c_1,\, c_2]-1 = \sum_{R\in \R}\eta_R \omega_R (c_1,c_2)R$,
we have
\bee
f(([c_1,\, c_2]-1)([c_3,\,c_4]-1)g)& = &\sum_{R\in\R}\eta_R \omega_R (c_1,c_2) f((R[c_3,\, c_4]-1)g)=\nn
&=& \sum_{R\in\R}f(\eta_R \omega_R (c_1,c_2)([c_3,\, c_4]-1)gR)= \nn
&=& f(([c_3,\, c_4]-1)g([c_1,\,c_2]-1)). \nonumber
\eee
Clearly, $g([c_1,\,c_2]-1)=([c_1,\,c_2]-1)g$ because $c_1,c_2 \in {\cal E}(g)$ and $\varkappa^2=1$.
\end{proof}

Due to Proposition \ref{???},
eqs. (\ref{1234}) and (\ref{3412}) imply
\be
\str ((A_{12}+B_{12})-(A_{34}+B_{34}))=
\str ( [c_3,\, c_4] B_{12} - [c_1,\, c_2] B_{34})
\ee
or
\bee
&&\str ((A_{12}+B_{12})-(A_{34}+B_{34}))=\nn
\label{summ}
&&\ =
\str (
\sum_{R\in \R: \ \omega_R(c_1,\, c_2)\ne 0,\ \omega_R(c_3,\, c_4)\ne 0}\eta_R
 ([c_3,\, c_4] \omega_R (c_1,c_2) - [c_1,\, c_2] \omega_R (c_3,c_4))Rg).
\eee
Consider one summand in eq. (\ref{summ})
\be\label{434}
I_R := ([c_3,\, c_4] \omega_R (c_1,c_2) - [c_1,\, c_2] \omega_R (c_3,c_4))Rg.
\ee
Rewrite $I_R$ using transformation defined in eqs. (\ref{trans1}) -- (\ref{trans2}):
\bee
 c_3 &=&c_3^R + U^R_{31} c_1+ U^R_{32} c_2, \quad \text{where } c_3^R\in Z_R,
\\
 c_4 &=&c_4^R + U^R_{41} c_1+ U^R_{42} c_2, \quad \text{where } c_4^R\in Z_R.
\eee
Note that now $\det U^R\ne 0$ since $\omega_R (c_3,c_4)\ne 0$.

Express all the terms in the right-hand side of eq.(\ref{434})
by means of $c_1$, $c_2$, $c_3^R$ and $c_4^R$:
\bee\label{90}
[c_3,\, c_4] \omega_R (c_1,c_2)Rg &=&\nn
&=&[c_3^R,\, c_4^R] \omega_R (c_1,c_2)Rg+\nn
\label{91}
&+&[c_3^R,\, U^R_{41} c_1+ U^R_{42} c_2] \omega_R (c_1,c_2)Rg+\nn
\label{92}
&+&[U^R_{31} c_1+ U^R_{32} c_2,\, c_4^R] \omega_R (c_1,c_2)Rg+\nn
\label{93}
&+&(U^R_{31}U^R_{42}-U^R_{32}U^R_{41})[c_1,\, c_2] \omega_R (c_1,c_2)Rg
\label{94}
\eee
\bee\label{95}
[c_1,\, c_2] \omega_R (c_3,c_4)Rg &=&\nn
&=&(U^R_{31}U^R_{42}-U^R_{32}U^R_{41})[c_1,\, c_2] \omega_R (c_1,c_2)Rg.
\label{96}
\eee

Since $c_3^R, c_4^R \in {\cal E}(Rg)$, Proposition \ref{ah-ah}
shows that
\bee
\str (([c_3^R,\, c_4^R] +
[c_3^R,\, U^R_{41} c_1+ U^R_{42} c_2] +
[U^R_{31} c_1+ U^R_{32} c_2,\, c_4^R] )\omega_R (c_1,c_2)Rg)\equiv 0.
\eee
So, from eqs. (\ref{90}) -- (\ref{95}) it follows that
\be
\str (I_R)\equiv 0
\ee
and
eq. (\ref{eq}) is proven.

Consider another four elements of ${\cal E}(g)$:
\bee
c_1'=\frac 1 {\sqrt 2}(\mu c_1 + \nu c_3),
\quad c_2'=\frac 1 {\sqrt 2}(\frac 1 \mu c_2 + \frac 1 \nu c_4),\nn
c_3'=\frac 1 {\sqrt 2}(\mu c_1 - \nu c_3),
\quad c_4'=\frac 1 {\sqrt 2}(\frac 1 \mu c_2 - \frac 1 \nu c_4).
\eee
Clearly, $\omega(c_i',c_j')$ is in a normal form and the relation
(\ref{eq}) holds for $c_i$ replaced by $c_i'$:
\be
\str (([c_1',\, c_2']- 1)g) \equiv \str (([c_3',\, c_4']- 1)g),
\ee
which implies, when eq. (\ref{eq}) is taken in account,
\be\label{eq1}
\frac \mu \nu \str ( [c_1,\, c_4]g) + \frac \nu \mu\str ([c_3,\, c_2]g) \equiv 0
\text { for arbitrary nonzero $\mu, \nu \in \mathbb C$}.
\ee
So
\be
\str ( [c_1,\, c_4]g) \equiv str ([c_2,\, c_3]g) \equiv 0.
\ee

Analogously, considering
\bee
&& c_1''=\frac 1 {\sqrt 2}(\mu c_1 + \nu c_4),
\quad c_2''=\frac 1 {\sqrt 2}(\frac 1 \mu c_2 - \frac 1 \nu c_3),\nn
&& c_3''=\frac 1 {\sqrt 2}(\mu c_1 - \nu c_4),
\quad c_4''=\frac 1 {\sqrt 2}(\frac 1 \mu c_2 + \frac 1 \nu c_3)
\eee
we see that
\be
\str ( [c_1,\, c_3]g) \equiv str ([c_2,\, c_4]g) \equiv 0.
\ee
This finishes the proof of Proposition \ref{state}
and Theorem \ref{th5}.
\end{proof}

\section{The number of independent $\varkappa$-traces on $ H_{1, \,\eta}(G)$}

\subsection{Main theorems}

\begin{theorem}\label{main0}
{\it The dimension of the space of $\varkappa$-traces on the
superalgebra $ H_{1, \,\eta}(G)$ is equal to the number of conjugacy
classes of elements without eigenvalue $\varkappa$ belonging to the
symplectic reflection group $ {G} \subset End(V)$.
Each central function on conjugacy classes of elements without eigenvalue $\varkappa$ belonging to the
symplectic reflection group $ {G} \subset End(V)$
can be uniquely extended to a $\varkappa$-trace on $ H_{1, \,\eta}(G)$.
}
\end{theorem}

\indent\begin{proof}
This Theorem follows from Theorem \ref{th6} (see below), Theorem \ref{even1}, and Theorem \ref{th5}.
\end{proof}

Clearly, Theorem \ref{main0} is equivalent to the following theorem.

\begin{theorem}\label{main1}
{\it Let the symplectic reflection group $ {G} \subset End(V)$
have
$T_G$ conjugacy classes without eigenvalue $1$ and $S_G$
conjugacy classes without eigenvalue $-1$.

Then the superalgebra
 $ H_{1, \,\eta}(G)$ possesses
$T_G$ independent traces and $S_G$ independent supertraces.
}
\end{theorem}

\begin{theorem} \label{th6}
{\it
Every $\varkappa$-trace on the algebra $ {\G} $
satisfying the equation
\bee\label{GLC}
\str ([c_1,\,c_2] \sigmama )=0\qquad \mbox{ for any } g\in  {G}  \mbox{
with }E(g)\neq 0 \mbox{ and any } c_1,c_2 \in {\cal E}(g),
\eee
can be uniquely extended to an even $\varkappa$-trace on $ H_{1, \,\eta}(G)$.}
\end{theorem}

For proof of Theorem \ref{th6}, see the rest of this section and Appendices.

The proof of Theorem \ref{th6} was published in \cite{KV} for the case
of supertraces (i.e., $\varkappa =-1$) on the superalgebra of observables of
Calogero model (i.e., $G=A_n$) and in \cite{KT} for the case $ H_{1, \,\eta}(G)$, where
the group $G$ is a finite group generated by a root system in $\mathbb R^N$.

Here we chose definitions of symbols such that
the rest of this section and Appendices
coincide almost literally with
analogous parts of \cite{KT} (and of \cite{KV}, if we change $\sigma \in S_N$ to
$g\in G\subset Sp(2N)$).

\subsection{The $\varkappa$-trace of General Elements}\label{sec5}

Proposition \ref{tomin}
does not prove Theorem \ref{th6} because the resulting values
of $\varkappa$-traces may
a priori depend on the sequence of step operations used and may
in principle impose additional constraints on
the values of $\varkappa$-trace on $\G$.

Below we prove that the value of $\varkappa$-trace does not depend
on the sequence of step operations used. We use the following inductive procedure:

{\bf $(\star)$}
{\it
Let $F\defeq P(b_{I})g\in \HH$, where $P$ is an even monomial such that $\deg P=2k$,
$b_I\in {\mathfrak B}_g$
and $g\in  {G} $.
Assuming that a $\varkappa$-trace is well defined for all elements of $\HH$ lesser
than  $F$ relative to the ordering from Definition \ref{ordering}, we prove that $\str (F)$ is defined also
without imposing any additional constraint on the solution of the Ground Level Conditions.
}

The central point of the proof is consistency conditions
(\ref{c1}), (\ref{c2}) and (\ref{c3}) proved in Appendices \ref{appb} and \ref{appc}.

   Assume that the  Ground Level Conditions hold.
   The proof of Theorem \ref{th6} will be given in a constructive way by
   the following double induction procedure, equivalent to ($\star$):                                                                     %

   {\bf (i)} Assume that
   $$
   \str \left([b_I , P_p (a) \sigmama
   ]_\varkappa \right) =0 \ \ \mbox{ for any $P_p (a),$ $\sigmama $ and $I$
   provided $b_I \in {\mathfrak B}_\sigmama$}
   $$
   and
   $$
   \begin{array}{l}
   \mbox{$\lambda (I) \neq \varkappa$; $p\leqslant k\,$  or}\\
   \mbox{$\lambda (I) =\varkappa$, $E(\sigmama  )\leqslant l$, $p\leqslant k\,$ or}\\
   \mbox{$\lambda (I) =\varkappa$; $p\leqslant k-2$}\,,
   \end{array}
   $$
   where $P_p (a)$ is an arbitrary degree $p$ polynomial in
   $a_{i} $ and $p$ is odd. This implies that there exists
   a unique extension of the $\varkappa$-trace such that the same is true for
   $l$ replaced with $l+1$.

   {\bf (ii)} Assuming that
   $\str \left( b_I  P_p (a) \sigmama - \varkappa P_p (a) \sigmama b_I \right)=0$
   for any $P_p (a)$, $\sigmama $ and
   $b_I \in {\mathfrak B}_\sigmama$,  where $p\leqslant k$,
   one proves that there exists a unique extension of the $\varkappa$-trace  such
   that the assumption {\bf (i)} is true for $k$ replaced with $k+2$ and $l=0$.

   As a result, this inductive procedure uniquely extends any solution of the
   Ground Level Conditions
   to a $\varkappa$-trace on the whole $H_{ {G} } (\eta )$.
   (Recall that the $\varkappa$-trace of any odd element of $H_{ {G} } (\eta )$ vanishes
   because the $\varkappa$-trace is even.)

It is convenient to work with the exponential generating
functions
\be
\label{gf}
\Psi_g   (\mu )=
\str \left ( e^S \sigmama  \right )\,,\mbox{ where }
S= \sum_{L=1}^{2N} (\mu^{L } b_{L} )\,,
\n
\ee
where $g $ is a fixed element of $ {G} $, $b_L \in {\mathfrak B}_g \,,$ and
$\mu^{L } \in {\mathbb C} $ are independent parameters.

The indices $I,J$ are raised and lowered with the help of the symplectic
forms $ {\cal C}^{IJ}$ and $ {\cal C}_{IJ}$ (see eq. (\ref{calc})):
\be\label{rise}
\mu_I=\sum_J{\cal C}_{IJ}\mu^J\,,\qquad \mu^I=\sum_J\mu_J {\cal C}^{JI}\,;
\qquad \sum_M{\cal C}_{IM}{\cal C}^{MJ}=-\delta_I^J\,.
\ee

By differentiating eq. (\ref{gf})
$n$ times with respect to
$\mu^{L }$ at $\mu=0$ one obtains a $\varkappa$-trace of an arbitrary polynomial
of $n$-th degree in $b_L$
as a coefficient of
$\sigmama $, up to polynomials of lesser degrees.
In these terms, the induction on the degree of
polynomials is equivalent to the induction on the homogeneity degree in
$\mu$ of the power series expansions of $\Psi_g  (\mu )$.

As a consequence of general properties of the $\varkappa$-trace,
the generating functions $\Psi_g  (\mu )$ must be $ {G} $-covariant:
\bee\label{S_N}
\Psi_{\tau g \tau^{-1}}(\mu)=\Psi_g
(\tilde{\mu})\,,
\eee
where the $ {G} $-transformed parameters are of the form
\bee
\label{base}
\tilde{\mu}^I=\left({\mathfrak M}(\tau g \tau^{-1})
{\mathfrak M}^{-1}(\tau)\Lambda^{-1}(\tau){\mathfrak M}(\tau)
{\mathfrak M}^{-1}( g )\right)^I_J {\mu}^J
\eee
and matrices ${\mathfrak M}(g )$ and $\Lambda(g)$ are defined below by eqs.
(\ref{frm}) and (\ref{eigmat}).

Let ${\mathfrak M}(g)$ be the matrix of the map
${\mathfrak B}_{\bf 1}\longrightarrow
{\mathfrak B}_g$, such that
\bee\label{frm}
b_I=\sum_{i} {\mathfrak M}^{i}_I(g)\, a_{ i}\,.
\eee
Obviously, this map is invertible.
Using the matrix notation one can
rewrite (\ref{basis}) as
\bee\label{eigmat}
\sigmama (b_I) =\sum_{J=1}^{2N} \Lambda_I^J(g)\,
 b_J,
\eee
where the matrix $(\Lambda_I^J)$ is diagonal, namely, $\Lambda_I^J=\delta_I^J\lambda_I$.

The necessary
and sufficient conditions for the existence of an even $\varkappa$-trace are the
$ {G} $-covariance conditions (\ref{S_N}) and the condition
\be\label{start}
\str \left(
\lbrack
b_L  , e^S
\sigmama  \rbrack _\varkappa\right)=0\qquad
\mbox{for any $ g $ and $L$} \,,
\ee
or, equivalently, taking in account that linear function $\str$
is an even $\varkappa$-trace,
\be\label{start1}
\str \left(
b_L  e^ S
\sigmama  - \varkappa e^S
\sigmama b_L \right)=0\qquad
\mbox{for any $ g $ and $L$} \,.
\ee

\subsection{General relations}\label{genrel}

To transform eq. (\ref{start1}) to a form convenient for the proof,
we use the following
two general relations true for arbitrary operators $X$ and $Y$ and
 parameter $\mu \in {\mathbb C}$:
\be\label{r1}
X \exp(Y+\mu X)=\frac {\partial }{\partial \mu} \exp  (Y+\mu X )+
\int \,t_2 \, \exp (t_1 (Y+\mu X))[X,Y]  \exp (t_2 (Y+\mu X))D^1t ,
\ee
\be\label{r2}
 \exp (Y+\mu X)X=\frac {\partial }{\partial \mu} \exp  (Y+\mu X )-
\int \,t_1 \, \exp (t_1 (Y+\mu X))[X,Y]  \exp (t_2 (Y+\mu X))D^1t
\ee
with the convention that
\be\label{t}
 D^{n-1}t=\delta (t_1 +\ldots +t_n -1)\theta (t_1 )\ldots \theta (t_n )
dt_1 \ldots dt_n \,.
\ee

The relations (\ref{r1}) and (\ref{r2}) can be derived with the help of
partial integration (e.g.,  over $t_1$) and the following formula
\be\label{d}
\frac {\partial }{\partial \mu} \exp  (Y+\mu X ) =
\int \,  \exp (t_1 (Y+\mu X)) X  \exp (t_2 (Y+\mu X))D^1 t\,
\ee
which can be proven by expanding in power series.  The well-known formula
\be
\label{r3}
[X, \exp (Y)]=
\int \,  \exp (t_1 Y)[X,Y]  \exp (t_2 Y)D^1 t
\ee
is a consequence of eqs. (\ref{r1}) and (\ref{r2}).

With the help of eqs. (\ref{r1}),  (\ref{r2}) and (\ref{basis1}) one rewrites
eq. (\ref{start1}) as
\be\label{nm1}
(1-\varkappa \lambda_L )\frac {\partial }{\partial \mu^L }\Psi_g   (\mu )=
\int
\,(-\varkappa \lambda_L t_1 -t_2 ) \str \ig (
\exp (t_1 S)[b_L ,S]\,\exp (t_2 S)\sigmama \ig )\, D^1 t\,.
\ee
This condition should be true for any $\sigmama $ and $L$ and plays the central
role in the analysis in this section.
Eq. (\ref{nm1}) is an overdetermined system of linear equations
for $\str$; below we show that it has the only solution extending any fixed solution
of the Ground Level Conditions.

There are two essentially distinct cases,  $\lambda_L \neq \varkappa $ and
$\lambda_L =\varkappa $. In the latter case, the eq. (\ref{nm1}) takes the form
\be\label{m1}
0=\int \, \str \ig (
\exp (t_1 S)[b_L ,S]\,exp (t_2 S) \sigmama  \ig )D^1 t\,,\qquad \lambda_L  =\varkappa \,.
\ee

In Appendix \ref{appb} we prove by induction that eqs. (\ref{nm1})
and (\ref{m1}) are consistent in the following sense:
\bee\label{c1}
\!\!\!\!\!\!\!\!\!\!\!\!
(1 -\varkappa \lambda_K )\frac {\partial }{\partial \mu^K }\int\,(-\varkappa \lambda_L t_1 -t_2 )
\str \ig (
\exp (t_1 S)[b_L ,S]\,\exp (t_2 S) \sigmama
\ig )D^1 t -(L \leftrightarrow K ) \equiv 0 & & \\
\mbox{for }\lambda_L \neq \varkappa , \ \lambda_K  \neq  \varkappa  & & \nonumber
\eee
and
\be
\label{c2}
(1 -\varkappa \lambda_K )\frac {\partial }{\partial \mu^K }\int \, \str \ig (
\exp (t_1 S)[b_L ,S]\,\exp (t_2 S)\sigmama \ig )D^1 t \equiv 0
\mbox{ for }   \lambda_L=\varkappa .
\ee
Note that this part of the proof is quite general and does not depend on a
concrete form of the commutation relations  in eq.
(\ref{rel}).

By expanding the exponential $e^S$ in eq. (\ref{gf}) into power series in $\mu
^K$ (equivalently $b_K$) we conclude that eq. (\ref{nm1}) uniquely
reconstructs the $\varkappa$-trace of monomials containing $b_K$ with
$\lambda_K\neq \varkappa $
(i.e., {\it regular monomials}) in terms of $\varkappa$-traces of some
lower degree polynomials. Then the consistency conditions (\ref{c1}) and (\ref{c2})
guarantee that eq. (\ref{nm1}) does not impose any additional conditions on
the $\varkappa$-traces of lower degree polynomials and allow one to represent the
generating function in the form
\bee\label{ex1}
\Psi_g  &=& \Phi_g (\mu)+\\
 &+&
\sum_{L:\,\lambda_L \neq \varkappa }
\int_0^1
\frac {\mu_L d\tau} {1-\varkappa \lambda_L}\int D^1 t\,(-\varkappa \lambda_L t_1 -t_2 ) \str \ig (
e^{t_1 (\tau S^{\prime\prime}+S^\prime)}[b_L ,(\tau S^{\prime\prime}+S^\prime)]
\,e^{ t_2 (\tau S^{\prime\prime}+S^\prime)}\sigmama \ig )\, ,
\nonumber
\eee
where we
introduced the generating functions $\Phi_g $ for the $\varkappa$-trace of
{\it special polynomials}, i.e.,
the polynomials depending only on $b_L$ with
$\lambda_L=\varkappa $, i.e., $b_L \in {\cal E}(g)$:
\be\label{gff}
\Phi_g   (\mu )\defeq
\str \left ( e^{ S^\prime} \sigmama  \right ) =
 \Psi_g  (\mu)\ig |_
{(\mu^I=0 \mbox{ {\footnotesize if} } \lambda_I \neq \varkappa )}
\n
\ee
and
\be
\label{spr}
S^\prime = \sum_{L:\,b_L \in {\mathfrak B}_\sigmama ,\,\lambda_L=\varkappa}
 (\mu^{L } b_L); \qquad S^{\prime\prime}=S-S^\prime\,.
\ee
The relation (\ref{ex1}) successively expresses the $\varkappa$-trace of
higher degree regular
polynomials via the $\varkappa$-traces of lower degree polynomials.

One can see that the arguments above prove the inductive
hypotheses {\bf (i)} and {\bf (ii)} for the particular case where the
polynomials $P_p (a)$ are regular and/or $\lambda_I \neq \varkappa $.  Note that for
this case the induction {\bf (i)} on the grading $E$
 is trivial: one simply proves that the degree of the
polynomial can be increased by two.

Let us now turn to a less trivial case of the special polynomials:
\be\label{startprime}
\str \left (b_I   e^{S^\prime} \sigmama - \varkappa e^{S^\prime}\sigmama b_I  \right )=0\,,
\mbox{ where } \lambda_I =\varkappa \,.
\ee
This equation implies
\be\label{startprime1}
\str \left ([b_I,\,   e^{S^\prime}]\sigmama  \right )=0\,,
\mbox{ where } \lambda_I =\varkappa \,.
\ee

Consider the part of $\str \left ([b_I  , \exp S^\prime ] \sigmama  \right )$
which is of degree $k$ in $\mu$ and let $E(g)=l+1$.
By eq. (\ref{m1}) the conditions (\ref{startprime1}) give
\be
\label{m1prime} 0=
\int \,
\str \left( \exp (t_1 S^\prime)[b_I ,S^\prime]\, \exp (t_2 S^\prime) \sigmama
\right) D^1 t\,.  \n
\ee

Substituting $[b_I ,S^\prime]=\mu_I +  \sum_M f_{IM}\mu^M$,
where the quantities $f_{IJ}$ and $\mu_I$ are defined
in eqs. (\ref{calc}), (\ref{37}) and (\ref{rise}), one can rewrite eq. (\ref{m1prime})
in the form
\bee\label{formprime}
\mu_I \Phi_g (\mu ) &=&
- \int
\str \bigg ( \exp (t_1 S^\prime)\sum_M f_{IM}\mu^M\, \exp (t_2 S^\prime) \sigmama
\bigg) D^1 t\,.
\eee

Now we use the inductive hypothesis {\bf (i)}.
The integrand in eq. (\ref{formprime}) is a $\varkappa$-trace of
a  polynomial of degree $\leqslant k-1$  in the $a_{\alpha\,i}$ in the sector of  degree $k$
polynomials in $\mu$, and $E(f_{IM}g)=l$.
Therefore one can use the inductive hypothesis {\bf
(i)} to obtain the equality
$$
\int \str \ig (  \exp (t_1 S^\prime)\sum_M f_{IM}\mu^M\,  \exp (t_2 S^\prime)
\sigmama  \ig )D^1t =
\int \, \str \ig (   \exp (t_2 S^\prime) \exp (t_1 S^\prime)
\sum_M f_{IM}\mu^M \sigmama  \ig )D^1 t,
$$
where we used
that
$\str (S^\prime F \sigmama )$ $=$ $ \varkappa \str (F
\sigmama  S^\prime) $= $\str (F  S^\prime \sigmama )$ by definition of
$S^\prime$.

As a result, the inductive hypothesis allows one to transform eq. (\ref{startprime})
to the form:
\be
\label{p9}
X_I =0,\text{~~where~~} X_I  \defeq   \mu_I \Phi_g (\mu )
+ \str \bigg( \exp (S^\prime ) \sum_Mf_{IM}\mu^M\sigmama  \bigg) \,.
\ee

By differentiating this equation with respect to $\mu^J$ one obtains after
symmetrization
\be
\label{p10}
\frac {\partial }{\partial \mu^J}   \left(
\mu_I
\Phi_g  (\mu )\right)
+(I\leftrightarrow J )=-
\int \str \ig (e^{t_1 S^\prime } b_Je^{t_2 S^\prime }
\sum_M f_{IM}\mu^M \sigmama  \ig )D^1 t +(I\leftrightarrow J ).
\ee

An important point is that the system of equations (\ref{p10}) is equivalent
to the original equations (\ref{p9}) except for the ground level part
$\Phi_g  (0)$.
This can be easily seen from the simple fact that the general solution
of the system of equations
for
entire functions
$X_I(\mu)$
$$
\frac {\partial }{\partial \mu^J} X_I(\mu) + \frac {\partial }{\partial \mu^I} X_{J}(\mu) =0
$$
is of the form
$$X_I(\mu)=X_I(0)+\sum_{J}c_{IJ}\mu^J$$
where
$X_I(0)$ and
$c_{IJ}$=$-c_{JI}$
are some constants.

The part of eq. (\ref{p9}) linear in $\mu$ is however equivalent to the
Ground Level Conditions
analyzed in Section \ref{anal1}.
Thus, eq. (\ref{p10}) contains all information on eq. (\ref{mm})
additional to the Ground Level Conditions.
For this reason, we will from now on analyze
equation (\ref{p10}).

Using again the inductive hypothesis
we move $b_I$ to the left and to
the right of the right hand side of eq. (\ref{p10})
with weights equal to $\frac 1 2$ each to get
\bee\label{p11}
&{}&
\frac {\partial }{\partial \mu^J}\mu_I \Phi_g  (\mu )+(I\leftrightarrow J )=
-\frac 1 {2} \sum_{M}\str \ig (  \exp (S^\prime )\{b_J ,f_{IM}\}\mu^M \sigmama \ig )-\nn
&{}&
-\frac {1}{2}\int \,\sum_{L,M}(t_1 -t_2 ) \str \ig (\exp (t_1 S^\prime )
F_{JL}\mu^L \exp (t_2 S^\prime ) f_{IM}\mu^M \sigmama  \ig )D^1 t
+ (I\leftrightarrow J )   \,.
\eee
The terms with the factor  $t_1 -t_2 $
vanish as is not difficult to show, so
eq. (\ref{p11}) reduces to
\be\label{dm1}
L_{IJ}\Phi_g  (\mu )=  -\frac {1}{2}
R_{IJ} (\mu )\,,
\ee
where
\bee\label{RIJ}
R_{IJ} (\mu )=\sum_{M}
\str \ig (  \exp (S^\prime )\{b_J ,f_{IM}\}\mu^M \sigmama  \ig )
+(I\leftrightarrow J )
\eee
and
\bee\label{LIJ}
L_{IJ}=\frac {\partial }{\partial \mu^J}\mu_I + \frac {\partial }{\partial \mu^I}\mu_J\,,
\eee
or, equivalently,
\bee\
L_{IJ}=\mu_I\frac {\partial }{\partial \mu^J} + \mu_J\frac {\partial }{\partial \mu^I}\,.
\eee

The differential operators $L_{IJ}$ satisfy the standard
commutation
relations of the Lie algebra
$\mathfrak{sp}(2E(g))$
\be\label{lcom}
[L_{IJ},L_{KL}]= - \left(
{\cal C}_{IK}L_{JL}+
{\cal C}_{IL}L_{JK}+
{\cal C}_{JK}L_{IL}+
{\cal C}_{JL}L_{IK} \right)
\,.
\ee
In Appendix \ref{appc} we show by induction  that this Lie algebra  $\mathfrak{sp}(2E(g))$
realized by differential operators is consistent
with the right-hand side of the basic relation (\ref{dm1}), i.e., that
\be\label{c3}
[L_{IJ},\,R_{KL}]-
[L_{KL},\,R_{IJ}]=    -\left(
{\cal C}_{IK}R_{JL}+
{\cal C}_{JL}R_{IK}+
{\cal C}_{JK}R_{IL}+
{\cal C}_{IL}R_{JK}     \right)
\,.
\ee

Generally, these consistency conditions guarantee that
eqs. (\ref{dm1})
express $\Phi_g  (\mu ) $ in terms of $R^{IJ}$ in the following
way
\bee
\label{ex2}
\Phi_g (\mu)&=&
\Phi_g (0)+\frac  {1}{8E(g)}\sum_{I,J=1}^{2E(g)} \int_0^1
\frac{dt}{t} (1-t^{2E(\sigmama  )})
(L_{IJ} R^{IJ})(t\mu ) \,,
\eee
provided
\be
\label{0}
R^{IJ}(0)=0\,.
\ee
The latter condition must hold for a consistency of eqs. (\ref{dm1})
since its left hand side vanishes at $\mu^I =0$. In the expression (\ref{ex2})
it guarantees that the integral over $t$  converges.
In the case under consideration the condition (\ref{0})
is met as follows from definition (\ref{RIJ}).

Taking  Lemma \ref{grad-1} and the explicit form (\ref{RIJ}) of $R_{IJ}$
into account  one concludes that eq.
(\ref{ex2}) uniquely expresses the $\varkappa$-trace
of special polynomials in terms of the
$\varkappa$-traces of polynomials of lower degrees or in terms of the $\varkappa$-traces of special
polynomials of the same degree multiplied by elements of $ {G} $ with a smaller value of $E$
provided that the $\mu$-independent term $\Phi_g (0)$ is an arbitrary solution
of {the Ground Level Conditions}.  This completes the proof of Theorem \ref{th6}. $\square$

\section{Non-deformed skew product $H_{1,0}(G)$ of the Weyl superalgebra and a finite
symplectic reflection group}\label{non}

Consider $H_{1,0}(G)$. It has the same number of traces and
supertraces as $ H_{1, \,\eta}(G)$ for an arbitrary $\eta$ and the
generating functions of these traces and supertraces are written below
explicitly. The algebra $ H_{1, \,0}(G)$ is the skew product $W_N*G$
of the Weyl superalgebra $W_N$
and the group algebra $\mathbb C [G]$ of the finite group $ {G}\subset Sp(2N) $ generated by a
system $\R\subset G$ of symplectic reflections. Algebras of this type,
and their generalizations, were considered in \cite{pass}.

Because the Weyl superalgebra $W_N$ is simple, the algebras $H_{1,0}(G)=W_N*G$  are also simple
(see \cite{pass}, p. 48, Exercise 6). This is a way to augment the stock of known simple
associative (super)algebras with several (super)traces.

It is easy to find
the general solution of eqs. (\ref{all}), (\ref{nm1}) and (\ref{m1})
for the generating function
of $\varkappa$-traces in the case $\eta=0$:

(1) If $g\in  {G} $ and $E(g)\ne 0$, then $sp(P(a_i)g)=0$ for any polynomial $P$.

(2) If $g\in  {G} $ and $E(g)= 0$, then $sp(g)$ is
an arbitrary central function on $ {G} $.

(3) Let $E(g)= 0$.
Let $S=\sum_{i} \mu^{i} a_{i}$\,,
$
\Psi(g,\mu,t):=\str( e^{tS}g),\;\Psi(g,\mu)=\str ( e^{S}g )
=\Psi(g,\mu,1)$.
Then
\begin{align}\label{nu0}
\str\left(  [a_{i},e^{tS}g]_{\varkappa}\right)   &  =\str \left(
t\omega_{ij}\mu^{j}e^{tS}g - e^{tS}a_{j}g
p_{i}^{j}\right)  ,\;\mbox{ where } p_{i}^{j}=(1-\varkappa g)_{i}^{j} \,.
\end{align}

Since $E(g)=0$, the matrix $(p_{i}^{j})$ is invertible, so eq.
(\ref{nu0}) gives

\begin{align*}
&  \frac d {dt}{\Psi}(g,\mu,t)=-\mu^{j}\omega_{ij}
q_{k}^{i}
\mu^{k}\Psi(g,\mu,t),
 \mbox{ where }
  q_{k}^{i}=\left(  \frac{1}{1-\varkappa g}\right)  _{k}^{i}=\frac{1}%
{2}\left(  \frac{\varkappa + g}{\varkappa - g}\right)  _{k}^{i}+\frac{1}%
{2}\delta_{k}^{i}.
\end{align*}%
So
\begin{align*}
\frac d {dt}{\Psi}(g,\mu,t)  &  =-Q(\mu)
{\Psi}(g,\mu,t),
\mbox{ where }
Q=\frac 1 2 \mu^i \mu^j \tilde\omega_{ij},
\mbox{ and } \tilde\omega_{ij} =
\omega_{ki} \left(\frac{\varkappa + g}{\varkappa - g}\right)_j^k
\end{align*}
and finally
\be\nonumber
\Psi(g,\mu)    = \exp
\left (
-\frac{1}{2}\mu^i \mu^j \omega_{ki}\left(\frac{\varkappa + g}{\varkappa - g}\right)_j^k
\right)
\str (\sigmama).
\ee

It is easy to check that the form $\tilde\omega_{ij}$
is symmetric.

\section{Lie algebras $\HH^L$ and Lie superalgebras $\HH^S$}

We can consider the space of associative algebra $\HH$ as
a Lie algebra $\HH^L$ with the brackets%
\footnote{
Recall that $[f,g]_\varkappa := fg - \varkappa^{\pi(f)\pi(g)}gf$.}
$[f,g]_{+1}=fg-gf$ for all $f,g\in \HH^L$.

We can also consider the space of associative algebra $\HH$ as
a Lie superalgebra $\HH^S$ with the brackets
$[f,g]_{-1}=fg-(-1)^{\pi(f)\pi(g)}gf$ for all $f,g\in \HH^S$.

S. Montgomery in \cite{SM} showed that it is possible
to construct simple Lie superalgebra $A^L$ from simple
associative superalgebra $A$ if the supercenter of $A$
satisfies some conditions. In particular, if the supercenter
is $\mathbb C$, then these conditions are satisfied.

\subsection{Center and supercenter of $\HH$.}

Let $\ce$ be the center of $\HH$,
i.e.,
$fz-zf=0$ for all $z\in\ce$ and for all $f\in \HH$.

Let $\se$ be the supercenter of $\HH$,
i.e.,
$fz-(-1)^{\pi(z)\pi(f)}zf=0$ for all $z\in\se$ and for all $f\in \HH$.
Clearly, $\se=\se_0\oplus\se_1$, where $\pi(\se_0)=0$ and $\pi(\se_1)=1$.
Evidently, $\se_0\subset \ce$.

\begin{theorem}\label{center}
$\ce=\se=\mathbb C$.
\end{theorem}

\begin{proof}
The first part of this Theorem, $\ce=\mathbb C$, is proven in \cite{BG}.
Further, $\se_0=\mathbb C$, and it remains to prove that
$\se_1=0$.

Suppose that there exists $z\in\se_1$. Then $z=\sum_{g\in G}P_g g$.
Consider $[z, b_g]_{-1}=zb_g+b_gz$ for all $b_g\in {\mathfrak B}_g$ for all $g\in G$.
One can see that $deg [z, b_g]_{-1}>deg z$ unless there exists the element
$K=-1$ in $G\subset Sp(2N)$ and
 $z=P_K K$.

If such element $K$ does not exist%
\footnote{
Clearly, $Kf=(-1)^{\pi(f)}fK$ for all $f\in\HH$, $\pi(K)=0$
and $K^2=1$. We call such element of $\HH$ {\it Klein operator}.
If Klein operator $K$ exists, then it defines the isomorphism of the
spaces of the traces and the supertraces on $\HH$ (see \cite{stek}).}
 then $z=0$
otherwise $zK\in \ce$ which also implies $z=0$ due to  $\pi(zK)=1$.
\end{proof}

Since $1\in G$, it follows from Theorem \ref{main0}
that there exists a supertrace $str_1$ such that $str_1(1)\ne 0$.
So, $[\HH^S,\,\HH^S]\cap \se=0$.

\subsection{Lie algebras and Lie superalgebras generated by $\HH$.}

\begin{definition}Set
\bee
\LLL &:=& [\HH^L,\,\HH^L]_{+1}\,/\left([\HH^L,\,\HH^L]_{+1}\cap \ce\right) \,;\\
\SSS &:=& [\HH^S,\,\HH^S]_{-1}.
\eee
\end{definition}

Now one can apply Theorem 3.8 of \cite{SM} (which generalizes the results of I.N.Herstein (see
\cite{H1}, \cite{H2})
to formulate the following statement

\begin{theorem}
{\it
If $\HH$ is a simple associative algebra, then

1) $\LLL$ is a simple Lie algebra,%

2) $\SSS$ is a simple Lie superalgebra.
}
\end{theorem}

\subsection{Bilinear forms on $\LLL$ and $\SSS$.}

If there exists a trace $tr_1$ on $\HH$ such that $tr_1(1)\ne 0$, then
$[\HH^L,\,\HH^L]_{+1}\cap \ce =0$.
If $tr(1)=0$ for any trace $tr$, then
$tr((f+\alpha)(g+\beta))=tr(fg)$ for any $f,g\in[\HH^L,\,\HH^L]_{+1}$ and
any $\alpha,\beta\in\mathbb C$.

So,
it is possible to define
a bilinear symmetric invariant form $B_{tr}$ on $\LLL$.

\begin{definition} Let $tr$ be a trace on $\HH$.

\noindent Let $\rho:\,[\HH^L,\,\HH^L]_{+1} \mapsto [\HH^L,\,\HH^L]_{+1}\,/\,[\HH^L,\,\HH^L]_{+1}\cap \ce$
be the natural projection.
Then
\be\label{bt}
B_{tr}(\rho(f),\,\rho(g)):= tr(fg) \ \texttt{for any } f,g\in [\HH^L,\,\HH^L]_{+1}
\ee
is a well defined bilinear form on $\LLL$.
\end{definition}

Define also a bilinear symmetric invariant form $B_{str}$ on $\SSS$.

\begin{definition} Let $str$ be a supertrace on $\HH$. Set
\be\label{bs}
B_{str}(f,\,g):= str(fg) \ \texttt{for any } f,g\in \SSS.
\ee
\end{definition}

To finish this section,
let us show that if $\HH$ is a simple associative algebra,
then maps  Eqs. (\ref{bt}) and (\ref{bs}) sending the (super)traces
into the spaces of bilinear invariant (super)symmetric forms are injections.

Suppose that $B_{str}\equiv 0$ for some supertrace $str$,
i.e., $str([a,b]_{-1}[c,d]_{-1})=0$ for any $a,b,c,d\in \HH$.
Hence, $str([[a,b]_{-1},c]_{-1}d)=0$ for any $a,b,c,d\in \HH$.
Since $\HH$ is simple, we have $[[a,b]_{-1},c]_{-1}=0$ for any $a,b,c\in \SSS$,
which contradicts to the simplicity of $\SSS$.

The proof for the traces is analogous.



\appendix

\section{Proof of consistency conditions.}

\subsection{Proof of consistency condition (\ref{c1}) for $\lambda\neq \varkappa$.}\label{appb}

Let parameters $\mu_1  \defeq   \mu^{K_1}$ and $\mu_2  \defeq
\mu^{K_2}$ be such that $\lambda_1 \neq \varkappa $ and $\lambda_2
\neq \varkappa $, where $\lambda_1  \defeq   \lambda_{K_1}$ and
$\lambda_2  \defeq \lambda_{K_2}$. Let $b^1  \defeq b_{K_1}$ and
$b^2  \defeq b_{K_2}$. Let us prove by induction that  conditions
(\ref{c1}) hold. To implement induction, we select the part of degree
$k$ in $\mu$ from eq. (\ref{nm1})  and observe that this part
contains a degree $k+1$ polynomial in $b_M$ in the left-hand side
of eq. (\ref{nm1}) while the part on the right hand side of the
differential version (\ref{nm1}) of eq. (\ref{start}), which is of the same degree
in $\mu$, has a degree $k-1$ as polynomial in $b_M$.

This happens because of
the presence of the commutator $[b_L ,S]$ which is a zero degree polynomial
due to the basic relations (\ref{rel}). As a result, the inductive hypothesis allows
us to use the properties of the $\varkappa$-trace provided that the
commutator $[b_L ,S]$ is always handled as the right hand side of eq. (\ref{rel}),
i.e., we are not
allowed to represent it again as a difference of the second-degree
polynomials.

Direct differentiation of Eq. (\ref{nm1}) with the help of eq. (\ref{d}) gives
\bee
\label{p1}
(1-\varkappa \lambda_2 )\frac {\p}{\p\mu_2 }
\int
\,(-\varkappa \lambda_1 t_1 -t_2 ) \str \ig (
e^{t_1 S} [b^1,S]\,e^{t_2 S}\sigmama \ig ) D^1 t
- \ig (1 \leftrightarrow 2 \ig )
&=&\nn
=\left(\int \,
(1-\varkappa \lambda_2 )(-\varkappa \lambda_1 t_1 -t_2 )
\str \left( e^{t_1 S} [b^1 ,b^2 ]\,e^{t_2 S}
\sigmama  \right )D^1 t      \,
- \ig (1 \leftrightarrow 2 \ig )  \right)
&+&\nn
+\igg (
       \int (1-\varkappa \lambda_2 ) (-\varkappa \lambda_1 (t_1 +t_2 )-t_3 )
   \str \ig (
        e^{t_1 S} b^2 e^{t_2 S} [b^1 ,S]\,e^{t_3 S}
       \ig ) D^2 t\,\,
              - \ig (1 \leftrightarrow 2 \ig )
\igg )
&+&
\nn
+ \igg ( \int (1-\varkappa \lambda_2 )
          (-\varkappa \lambda_1 t_1 -t_2 -t_3 )
  \str  \ig(
          e^{t_1 S} [b^1 ,S]\,e^{t_2 S} b^2
          e^{t_3 S}\sigmama
      \ig )D^2 t \,
         - \ig (1 \leftrightarrow 2 \ig )
\igg)\,.&{}&
\eee

We have to show that the right hand side of eq. (\ref{p1}) vanishes.  Let us
first transform the second and the third terms on the right-hand side of eq.
(\ref{p1}). The idea is to move the operators $b^2$ through the exponentials
towards the commutator $[b^1 ,S]$ in order to use then the Jacobi identity for the
double commutators.  This can be done in two different ways inside  the
$\varkappa$-trace so that one has to fix appropriate weight factors for each of
these processes.
Let the notation ${\overrightarrow A}$ and ${\overleftarrow
A}$ mean that the operator $A$ has to be moved from its position to the
right and to the left, respectively.

The correct weights turn out to be
\bee\label{p2}
D^2 t(-\varkappa \lambda_1 (t_1 +t_2 )-t_3 )b^2 \equiv
D^2 t(-\varkappa \lambda_1 -t_3 (1-\varkappa  \lambda_1 ))b^2=
\nn
= D^2 t\left (\igg (\frac { \lambda_1 \lambda_2}{1-\varkappa \lambda_2}-t_3 (1-\varkappa
 \lambda_1 )\igg )
\overrightarrow {b^2} +
\frac {-\varkappa \lambda_1 }{1-\varkappa \lambda_2} \overleftarrow {b^2} \right )
\n
\eee
and
\bee\label{p3}
D^2 t(-\varkappa \lambda_1 t_1 -t_2 -t_3 )b^2
\equiv D^2 t ((-\varkappa \lambda_1 +1) t_1 -1)b^2=
\nn
= D^2 t\left (
\igg (t_1 (1-\varkappa  \lambda_1 ) -\frac {1}{1-\varkappa \lambda_2} \igg )
\overleftarrow {b^2} -
\frac {-\varkappa \lambda_2 }{1-\varkappa \lambda_2} \overrightarrow {b^2} \right )
\n
\eee
for the second and third terms in the right hand side of eq. (\ref{p1}),
respectively.   Using eq. (\ref{r3})
along with the simple formula
\be\label{p4}
\int \, \phi (t_3 ,\ldots t_{n+1} )D^n t=
\int \, t_1\phi (t_2 ,\ldots t_n )D^{n-1} t
\n
\ee
we find that all terms which involve both $[b^1 ,S]$ and
$[b^2,S]$ pairwise
cancel after antisymmetrization $1\leftrightarrow 2$.

As a result, one is left with some terms involving double commutators which,
thanks to the Jacobi identities and antisymmetrization, are all reduced to
\be\label{1p5}
\int\,\ig (\lambda_1\lambda_2 t_1+t_2-t_1 t_2 (1-\varkappa \lambda_1)(1-\varkappa \lambda_2)\ig )
\str \ig (\exp (t_1 S)
[S,[b^1 ,b^2 ]]\exp (t_2 S) \sigmama \ig ) D^1 t\,.
\n
\ee
Finally, we observe that this expression can be equivalently rewritten in
the form
\be\label{p5}
\int \,
\ig (\lambda_1\lambda_2 t_1+t_2-t_1 t_2 (1-\varkappa \lambda_1)(1-\varkappa \lambda_2)\ig )
\left(\frac {\p}{\p t_1}-\frac {\p}{\p t_2} \right) \str \ig (\exp (t_1 S)
[b^1 ,b^2 ]\exp (t_2 S) \sigmama \ig )
D^1 t
\ee
and after integration by parts cancel the first term on the right-hand side
of eq. (\ref{p1}). Thus, we showed that eqs. (\ref{nm1}) are
compatible for the case $\lambda_{1,2}\neq \varkappa  $.

Analogously, we can show that eqs. (\ref{nm1}) are compatible with eq.
(\ref{m1}).  Indeed, let $\lambda_1 =\varkappa $, $\lambda_2 \neq \varkappa $. Let us prove
that
\be\label{p6}
\frac {\p}{\p\mu_2}\str \ig ([b^1 , \exp (S)]\sigmama  \ig )=0
\ee
provided the $\varkappa$-trace is well-defined for the lower degree polynomials.
The
explicit differentiation gives
\bee\label{p7}
\!\!\!\!\!\!\!\!\!\!\!\!\!\!\!\!\!\!
\frac {\p}{\p\mu_2}\str \ig ([b^1 , \exp (S)]\sigmama  \ig )&=&
\int \,\str \ig ([b^1 , \exp (t_1 S)b^2  \exp (t_2 S)]\sigmama  \ig
)D^1 t
=
\nn &=&
(1-\varkappa \lambda_2 )^{-1} \str \ig ([b^1 ,(b^2
 \exp (S) -\varkappa  \lambda_2  \exp ( S)b^2 )]\sigmama  \ig ) +\ldots
\eee
where dots denote some terms of the form $\str \ig( [b^1 , B]\sigmama \ig )$
involving more commutators inside $B$, which therefore amount to some
lower degree polynomials and vanish by the inductive hypothesis. As a result,
we find that
\bee\label{p8}
\!\!\!\!\!\!\!\!\!\!\!\!\!\!\!\!\!\!
\frac {\p}{\p\mu_2}\str \ig ([b^1 , \exp (S)]\sigmama  \ig )&=&
(1-\varkappa \lambda_2 )^{-1} \str \ig ((b^2 [b^1 , \exp (S)] -\varkappa  \lambda_2
[b^1 , \exp ( S)]b^2 )\sigmama  \ig )
+
\nn
&+& (1-\varkappa \lambda_2 )^{-1} \str \ig (([b^1 ,b^2 ] \exp (S) -\varkappa
\lambda_2  \exp ( S)[b^1 ,b^2 ])\sigmama  \ig )\,.
\eee
This expression vanishes by the inductive hypothesis, too.

\subsection{The proof of consistency conditions (\ref{c3}) (the case of special polynomials)}\label{appc}

In order to prove eq. (\ref{c3}) we use the inductive hypothesis {\bf (i)}.
In this Appendix we use  the convention that any
expression with the coinciding
upper or lower indices
are automatically symmetrized, e.g.,
$U^{II} \defeq  \frac  1 2 (U^{I_1 I_2}+U^{I_2 I_1})$.
In this Appendix, all the eigenvectors $b_I$  of $g$ belong to ${\cal E}(g)$.
The identity
\be\label{p12}
0 = \sum_{M}\str \ig (\ig [  \exp (S^\prime )
\{ b_I ,f_{IM}\}\mu^M ,b_J b_J\ig ]
\sigmama \ig ) -(I \leftrightarrow J)
\ee
holds
due to Lemma \ref{lemma4}
for all terms of degree $k-1$ in $\mu$
with $E(g) \leqslant l+1$
and
for all lower degree polynomials in $\mu$,
because one
can always move $f_{IJ}$ to $\sigmama $ in eq. (\ref{p12}) combining $f_{IJ}\sigmama$
into a combination of
elements of $ {G} $ analyzed in Lemma \ref{lemma4}.

Straightforward calculation of the commutator in the right-hand-side of eq.
(\ref{p12}) gives $0 = X_1+X_2+X_3$, where
\bee\label{1p13}
X_1&=&-\sum_{M,L}\int\,\str
          \left (   \exp (t_1 S^\prime ) \{b_J , F_{JL} \}
              \mu^L  \exp (t_2 S^\prime ) \{b_I ,f_{IM}\}\mu^M \sigmama
          \right) D^1 t
       -(I \leftrightarrow J)\,,\nn
X_2&=&\sum_M \str
          \left (  \exp (S^\prime )
              \ig \{\{b_J, F_{IJ}\},f_{IM}
              \ig\}  \mu^M \sigmama
          \right)-(I \leftrightarrow J)\,,\nn
X_3&=& \sum_M \str
              \left (  \exp (S^\prime )
                     \ig
\{b_I,\{b_J,[f_{IM},b_J]\}
                     \ig \}\mu^M \sigmama
               \right)-(I \leftrightarrow J)\,.
\eee
The terms of $X_1$ bilinear in $f$ cancel due to the
antisymmetrization ($I\leftrightarrow J$) and the inductive hypothesis
{\bf (i)}. As a result, one
can transform $X_1$ to the form
\bee\label{p13}
X_1
=
\left (- \frac  1 2
\left [ L_{JJ},\,R_{II}\right ] +2
\str \ig ( e^{S^\prime }
\{b_I ,f_{IJ}\}\mu_J \sigmama \ig )
\right)
-
(I \leftrightarrow J).
\eee

Substituting
$F_{IJ}={\cal C}_{IJ}+ f_{IJ}$ and $f_{IM}= [b_I,b_M] -{\cal C}_{IM}$
one transforms $X_2$ to the form
\bee\label{x2}
X_2
&=& 2{\cal C}_{IJ}R_{IJ}
-2 \left( \str
         \ig ( e^{S^\prime } \{b_J ,f_{IJ}\}\mu_I \sigmama
         \ig ) -(I \leftrightarrow J)
   \right)+Y,
\eee
where
\bee\label{Y}
Y=  \str \left ( e^{S^\prime }
              \ig \{ \{b_J, f_{IJ}\},[b_I,\,S^\prime]
              \ig\}   \sigmama
          \right)
    -(I \leftrightarrow J)\,.
\eee
Using that
\bee\label{S}
 \str
    \left (  \exp (S^\prime )
       \left[   P f_{IJ}Q,\,S^\prime
       \right]   \sigmama
    \right)
=0
\eee
provided the inductive hypothesis can be used,
one transforms $Y$ to the form
\bee\label{x2x2}
Y
\!\!
=
\!
\str
 \left ( e^{S^\prime }
   \left(-[f_{IJ}, (b_I S^\prime b_J + b_J S^\prime b_I )]
              - b_I [f_{IJ}, S^\prime] b_J -b_J [f_{IJ}, S^\prime] b_I
              + [f_{IJ}, \{b_I, b_J\}] S^\prime
    \right) \sigmama
 \right).\nn
\eee

Let us rewrite $X_3$ in the form $X_3=X_3^{s}+X_3^{a}$, where
\bee
X_3^s=
\frac  1 2 \sum_M \str
              \left ( e^{S^\prime }
               \ig (
                    \ig\{b_I,\{b_J,[f_{IM},b_J]\}\ig \} +
                    \ig\{b_J,\{b_I,[f_{IM},b_J]\}\ig \}
                \ig )
                          \mu^M \sigmama
               \right)
-
(I \leftrightarrow J)\,,\nn
X_3^a=
\frac  1 2 \sum_M \str
              \left ( e^{S^\prime }
               \ig (
                    \ig\{b_I,\{b_J,[f_{IM},b_J]\}\ig \} -
                    \ig\{b_J,\{b_I,[f_{IM},b_J]\}\ig \}
                \ig )
                          \mu^M \sigmama
               \right)
-
(I \leftrightarrow J)\,. \nonumber
\eee
With the help of the Jacobi identity
$[f_{IM},b_J]-[f_{JM},b_I]=[f_{IJ},b_M]$
one expresses $X_3^s$ in the form
$$
X_3^s=
\frac  1 2 \str
\left ( e^{S^\prime}
 \left (
    \{b_I,b_J\}[f_{IJ}, S^\prime] + [f_{IJ}, S^\prime]\{b_I,b_J\}
     +2 b_I[f_{IJ}, S^\prime]b_J + 2 b_J[f_{IJ}, S^\prime]b_I
  \right ) \sigmama
\right ).
$$
Let us transform this expression for $X_3^a$ to the form
\bee\label{100}
X_3^a=
\frac  1 2 \sum_M
\str
\left ( e^{S^\prime}
      \left[  F_{IJ},\, [f_{IM},b_J]\right ] \mu^M \sigmama
\right )       -(I \leftrightarrow J).
\eee

Substitute
$F_{IJ} ={\cal C}_{IJ}+ f_{IJ}$ and $f_{IM}= [b_I,b_M]-{\cal C}_{IM}$
in eq. (\ref{100}).
After simple transformations we find that $Y+X_3=0$.
From eqs. (\ref{p13}) and (\ref{x2}) it follows that
the right hand side of eq. (\ref{p12}) is
equal to
$$
\frac  1 2 ([L_{II},\, R_{JJ}]-[L_{JJ},\, R_{II}]) + 2{\cal C}_{IJ} R_{IJ}.
$$
This completes the proof of the consistency conditions
(\ref{c3}).


\end{document}